\documentclass[11 pt]{amsart}

\usepackage[english]{babel}
\usepackage[latin1]{inputenc}
\usepackage[T1]{fontenc}
\usepackage{graphicx}
\usepackage{amsmath}

\usepackage{amssymb}
\usepackage{latexsym}
\usepackage{amsthm}
\usepackage[all]{xy}
\usepackage{stmaryrd}
\usepackage{subfigure}
\usepackage{MnSymbol}
\usepackage{comment} 

\def\be{\begin{equation}}
\def\ee{\end{equation}}
\def\beq{\begin{eqnarray*}}
\def\eeq{\end{eqnarray*}}
\def\Z{\mathbb{Z}}

\newcommand{\pic}[3]{\parbox[c]{#1cm}{\includegraphics[scale=#2]{#3}}}
\newcommand{\picw}[2]{\parbox[c]{#1cm}{\includegraphics[width = #1cm]{#2}}}

\newcommand{\cerchio}{\includegraphics[width = .4 cm]{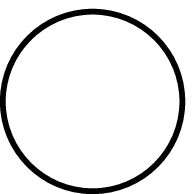}}
\newcommand{\teta}{\includegraphics[width = .4 cm]{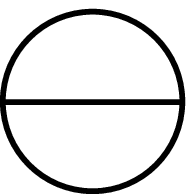}}
\newcommand{\tetra}{\includegraphics[width = .4 cm]{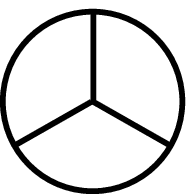}}
\newcommand{\gl}{{\rm gl}}
\newcommand{\ord}{{\rm ord}}

\newtheorem{theo}{Theorem}[section]
\newtheorem{cor}[theo]{Corollary}
\newtheorem{lem}[theo]{Lemma}
\newtheorem{prop}[theo]{Proposition}

\theoremstyle{definition}
\newtheorem{defn}[theo]{Definition}
\newtheorem{ex}[theo]{Example}
\newtheorem{quest}[theo]{Question}
\newtheorem{rem}[theo]{Remark}

\author{Alessio Carrega}
\address{Dipartimento di Matematica, Largo Pontecorvo 5, 56127 Pisa, Italy}
\email{carrega at mail dot dm dot unipi dot it}

\title[The Tait conjecture in $\#_g(S^1\times S^2)$]{The Tait conjecture in $\#_g(S^1\times S^2)$}

\begin{document}

\begin{abstract}
The Tait conjecture states that alternating reduced diagrams of links in $S^3$ have the minimal number of crossings. It has been proved in 1987 by M. Thistlethwaite, L. Kauffman and K. Murasugi studying the Jones polynomial. In \cite{Carrega_Tait1} the author proved an analogous result for alternating links in $S^1\times S^2$ giving a complete answer to this problem. In this paper we extend the result to alternating links in the connected sum $\#_g(S^1\times S^2)$ of $g$ copies of $S^1\times S^2$. In $S^1\times S^2$ and $\#_2(S^1\times S^2)$ the appropriate version of the statement is true for $\Z_2$-homologically trivial links, and the proof also uses the Jones polynomial. Unfortunately in the general case the method provides just a partial result and we are not able to say if the appropriate statement is true. 
\end{abstract}

\maketitle

\setcounter{tocdepth}{1}
\tableofcontents

\section{Introduction}

In the $19^{\rm th}$ century \cite{Tait}, P.G. Tait stated three conjectures about \emph{crossing number}, \emph{alternating links} and \emph{writhe number}. The \emph{crossing number} is the minimal number of crossings that a link diagram must have to represent that link. An \emph{alternating link} is a link that admits an \emph{alternating diagram}: a diagram $D$ such that the parametrization of its components $S^1\rightarrow D\subset D^2$ meets overpasses and underpasses alternately. All the conjectures have been proved before 1991. A diagram $D$ of a link in $S^3$ is said to be \emph{reduced} if it has no crossings as the ones in Fig.~\ref{figure:reducedDS3} (the blue parts cover the rest of the diagram). 

\begin{figure}[htbp]
\begin{center}
\includegraphics[scale=0.6]{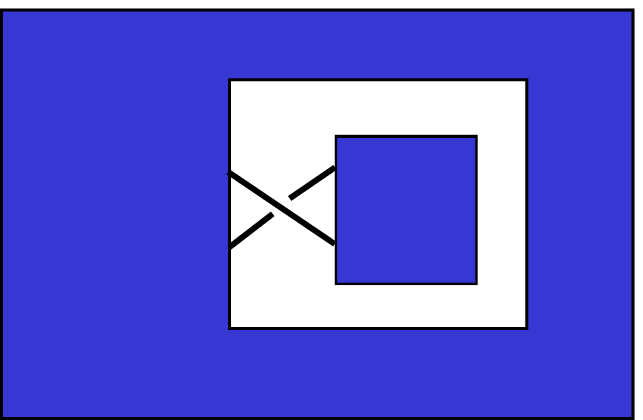}
\hspace{0.5cm}
\includegraphics[scale=0.6]{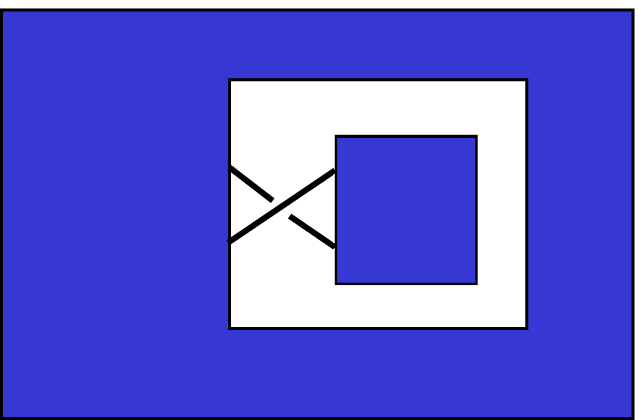}
\end{center}
\caption{Non reduced diagrams of links in $S^3$.}
\label{figure:reducedDS3}
\end{figure}

We focus just on one Tait conjecture, the following one: \emph{``Every reduced alternating diagram of links in $S^3$ has the minimal number of crossings''.} 

In 1987 M. Thistlethwaite \cite{Thistlethwaite}, L.H. Kauffman \cite{Kauffman_Tait} and K. Murasugi \cite{Murasugi1, Murasugi2} proved this conjecture. We are interested in the more general case of links in the connected sum $\#_g(S^1\times S^2)$ of $g\geq 0$ copies of $S^1\times S^2$ ($g=0$ means $S^3$ and $g=1$ means $S^1\times S^2$). The work is a direct extension of the case of links in $S^1\times S^2$ that has been studied by the author in \cite{Carrega_Tait1} giving a complete answer to the problem.

A link diagram in the 2-disk with $g$ holes,  $S_{(g)}$, defines a link in its thickening $H_g=S_{(g)}\times [-1,1]$ that is the handlebody of genus $g$, and hence in its double $\#_g(S^1\times S^2)$. We call \emph{H-decomposition} a decomposition of $\#_g(S^1\times S^2)$ as the union of two handlebodies of genus $g$ that intersect on the boundary. There is only one H-decomposition of $\#_g(S^1\times S^2)$ up to isotopies (Theorem~\ref{theorem:Heegaard_split_S1xS2}). Once a proper embedding of $S_{(g)}$ in a handlebody of the H-decomposition is fixed, every link in $\#_g(S^1\times S^2)$ can be described by a diagram in the punctured disk $S_{(g)}$. Thus we can still talk about crossing number and alternating diagrams (Definition~\ref{defn:alt_cr_num}). We define the \emph{crossing number} of a link $L\subset \#_g(S^1\times S^2)$ as the minimal number of crossings that a diagram must have to represent $L$ in some embedded disk with $g$ holes. A link diagram divides $S_{(g)}$ in 2-dimensional connected \emph{regions}. We say that a region of $S_{(g)}$ defined by $D$ is \emph{external} if it touches the boundary of $S_{(g)}$, otherwise it is \emph{internal}. We extend the notion of ``reduced'' as follows:
\begin{defn}\label{defn:reduced_g}
A link diagram $D \subset S_{(g)}$ is \emph{simple} if there is no disk $B$ embedded in $S_{(g)}$ whose boundary intersects $D$ exactly in one crossing (as the ones of Fig.~\ref{figure:reducedD_g}, the yellow square is the disk $B$), and $D$ has no crossings adjacent to two external regions (as the ones of Fig.~\ref{figure:reducedD2_g}), neither crossings adjacent twice to the same external region.
\end{defn}

\begin{figure}[htbp]
\begin{center}
\includegraphics[scale=0.55]{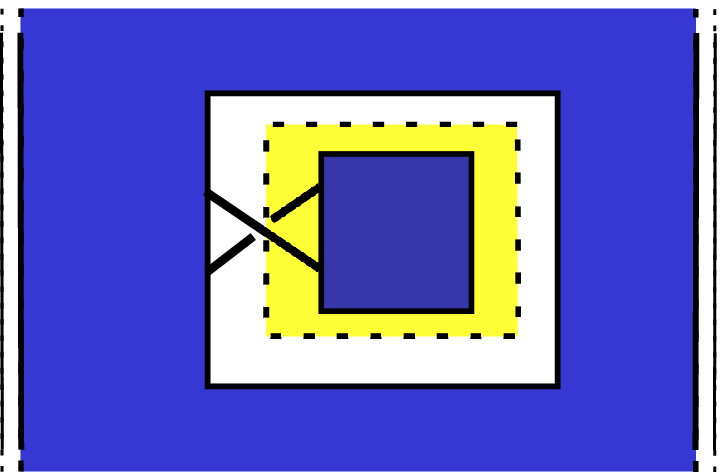} 
\hspace{0.5cm}
\includegraphics[scale=0.55]{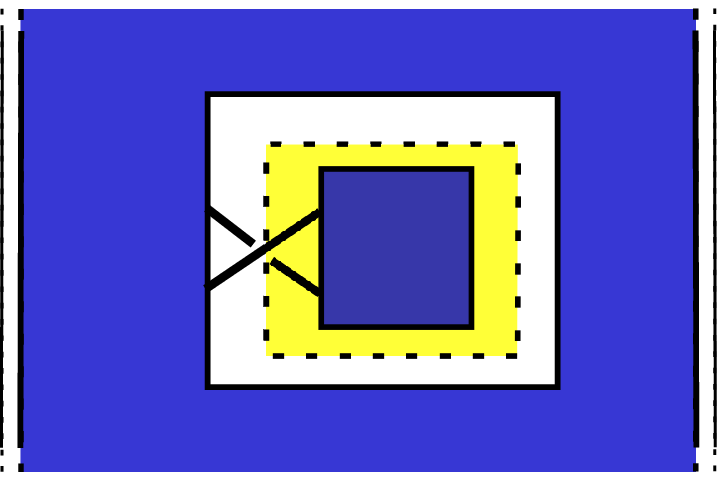}
\end{center}
\caption{Some non simple diagrams. We only show a
portion of $S_{(g)}$ homeomorphic to $(-1, 1)\times [-1,1]$. The blue parts cover the rest of the diagram and the yellow box is the embedded disk $B\subset S_{(g)}$ that intersects the diagram just in a crossing.}
\label{figure:reducedD_g}
\end{figure}

\begin{figure}[htbp]
\begin{center}
\includegraphics[scale=0.55]{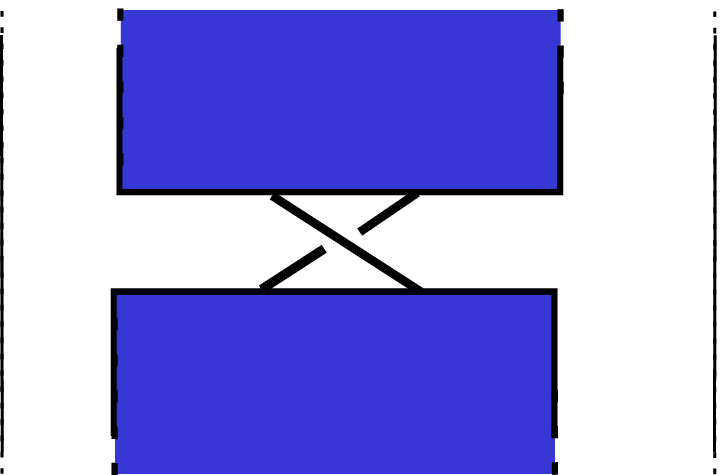} 
\hspace{0.5cm}
\includegraphics[scale=0.55]{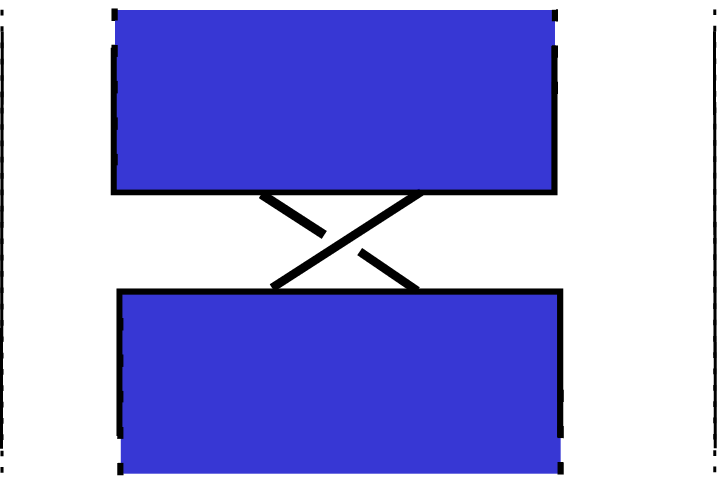}
\end{center}
\caption{More non simple diagrams. We only show a
portion of $S_{(g)}$ homeomorphic to $(-1, 1)\times [-1,1]$. The blue parts cover the rest of the diagram.}
\label{figure:reducedD2_g}
\end{figure}

Our definition of ``simple'' was designed to obtain the following theorem, which is the main result of the paper. Before stating it, we specify that we say that a link $L$ is $\Z_2$-\emph{homologically trivial} if its homology class $[L]\in H_1(S^1\times S^2; \Z_2)$ is trivial $[L]=0$ ($\Z_2 = \Z / 2\Z$). Furthermore we remind that being $\Z_2$-homologically trivial is equivalent to be the boundary of an embedded surface. 

\begin{theo}\label{theorem:Tait_conj_g}
Let $D\subset S_{(g)}$ be an alternating, simple diagram that represents a link $L\subset \#_g(S^1\times S^2)$ by a proper embedding $e$ of $S_{(g)}$ in a handlebody of the H-decomposition. Suppose that $L$ is non H-split (\emph{e.g.} a knot), $\Z_2$-homologically trivial and with homotopic genus $g$. Then for any diagram $D'\subset S_{(g)}$ that represents $L$ by an embedding $e'$ of $S_{(g)}$
$$
n(D) \leq n(D') + \frac{g-1}{2} ,
$$
where $n(D)$ and $n(D')$ are the number of crossings of $D$ and $D'$. In particular if $g\leq 2$, we have that $n(D)$ is the crossing number of $L$.
\end{theo}

Being \emph{non H-split} means that every diagram $D$ that represents the link by any embedded disk with $g$ holes is connected, namely it is so as a 4-valent graph, while being with \emph{homotopic genus} $g$ means that no diagram $D$ in a punctured disk $S_{(g)}$ of the link can be contained in a punctured disk with less than $g$ holes (Definition~\ref{defn:split_homotopic_genus}). The really interesting links in $\#_g(S^1\times S^2)$ respect these conditions.

The \emph{Jones polynomial} and the \emph{Kauffman bracket} are invariants that associate to each (oriented or framed) link $L\subset S^3$ a Laurent polynomial with integer coefficients $J(L), \langle L \rangle \in \Z[A,A^{-1}]$ (we use the variable $A= t^{-\frac 1 4} = \sqrt{\pm q^{\pm 1}}$ and we normalize it so that $J\left( \pic{0.8}{0.3}{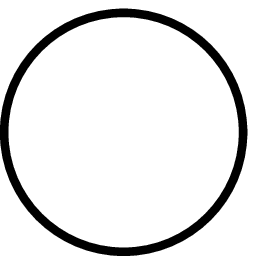} \right) = \left\langle \pic{0.8}{0.3}{banp.eps} \right\rangle = -A^2-A^{-2}$). 

The theorem of Thistelwaithe-Kauffman-Murasugi is one of the most notable applications of the Jones polynomial. In fact the proof of the classical case of links in $S^3$ is based on the \emph{breadth} $B(J(L))= J(\langle L \rangle)\in \Z$ of the Jones polynomial (or of the Kauffman bracket) that is an integer associated to each Laurent polynomial $f$: the difference between the biggest degree and the lowest degree of the non zero monomials of $f$.

The Jones polynomial, or more precisely the Kauffman bracket,  is also defined in $\#_g(S^1\times S^2)$ (see for instance \cite{Carrega-Martelli}, \cite{Carrega_Tait1}, \cite{Costantino2} or Section~\ref{section:Preliminaries_g}). For links in $S^1\times S^2$ ($g=1$) the Kauffman bracket is still a Laurent polynomial. When $g\geq 2$ it is a rational function and may not be a Laurent polynomial (see Example~\ref{ex:Kauff}). Therefore we need to extend the notion of ``breadth'' to rational functions $B(f/h)$, and we define it as the difference between the breadth of the numerator and the breadth of the denominator $B(f/h)= B(f)-B(h)$ (Definition~\ref{defn:breadth_g}).

Proceeding in a similar way as Thistelwaithe-Kauffman-Murasugi, we get the following result:

\begin{theo}\label{theorem:Tait_conj_Jones_g}
Let $D\subset S_{(g)}$ be a $n$-crossing, connected, alternating diagram of a $\Z_2$-homologically trivial link in $\#_g(S^1\times S^2)$ without crossings as the ones of Fig.~\ref{figure:reducedD_g} and not contained in a disk with less than $g$ holes ($g(D)=g$). Then
$$
B(\langle D \rangle) = 4n +4 -4g -4k ,
$$
where $k$ is the number of crossings adjacent to two external regions (as the ones of Fig.~\ref{figure:reducedD2_g}), neither adjacent twice to just one external region.
\end{theo}

We say that a link in $\#_g(S^1\times S^2)$ is alternating if it is represented by an alternating diagram in some embedded disk with $g$ holes. The previous theorem gives criteria to detect if a link in $\#_g(S^1\times S^2)$ is not alternating (Corollary~\ref{cor:conj_Tait_Jones_g}), and we show some examples (Example~\ref{ex:no_alt_g}).

The Kauffman bracket is also very sensitive to the $\Z_2$-homology class of the link. In fact the Jones polynomial of $\Z_2$-homologically non trivial links is always $0$ (Proposition~\ref{prop:0Kauff_g}).

The condition of Theorem~\ref{theorem:Tait_conj_g} of being $\Z_2$-homologically trivial is essential. In fact following the method of \cite[Subsection 3.1]{Carrega_Tait1} we can show that without that hypothesis the statement is false: the diagrams in Fig.~\ref{figure:no_Tait_g} are both alternating, simple and represent the same $\Z_2$-homologically non trivial knot $K\subset \#_g(S^1\times S^2)$ once fixed the embedding of the punctured disk, but they have a different number of crossings. 

\begin{figure}
\begin{center}
\includegraphics[scale=0.45]{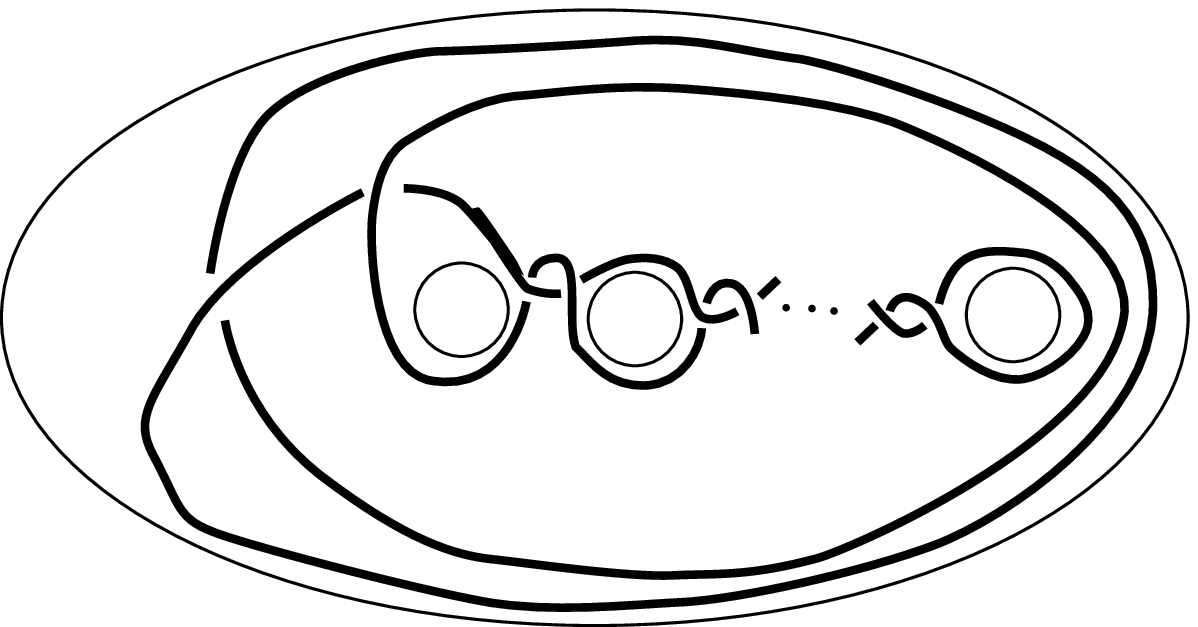} 
\hspace{0.5cm}
\includegraphics[scale=0.45]{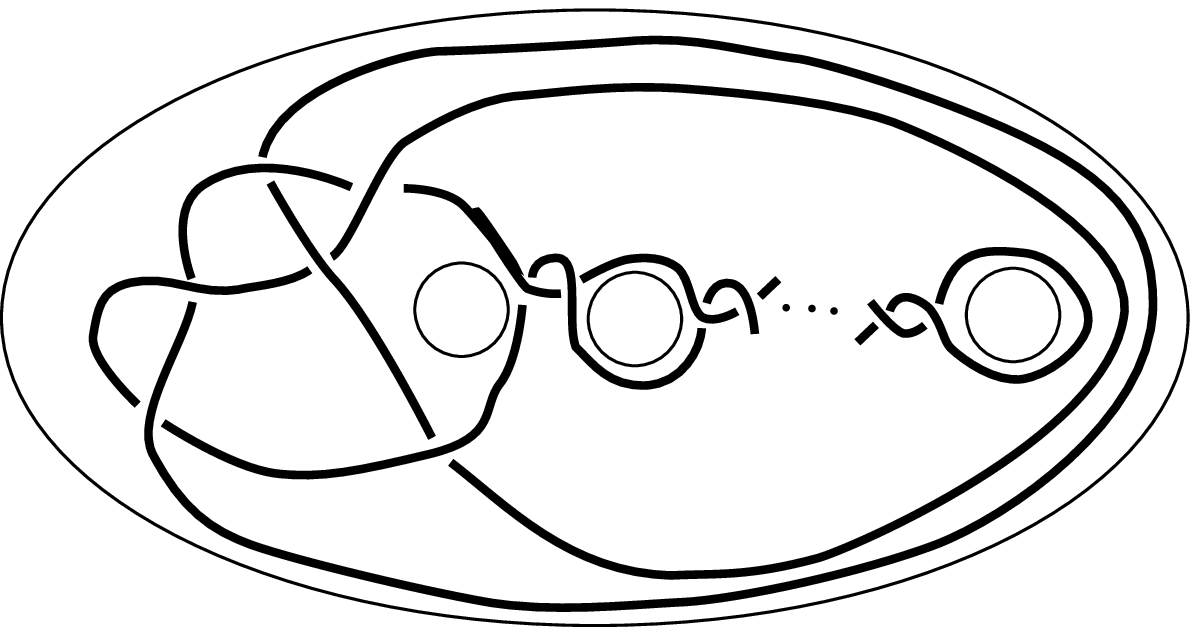}
\end{center}
\caption{Two alternating and simple diagrams of the same $\Z_2$-homologically non trivial knot which have different number of crossings.}
\label{figure:no_Tait_g}
\end{figure}

In $S^3$ all link diagrams with the minimal number of crossings are reduced. This statement has been discussed in \cite{Carrega_Tait1} for links in $S^1\times S^2$. Clearly the diagrams in $S_{(g)}$ with the minimal number of crossings do not have crossings as the ones of Fig.~\ref{figure:reducedD_g} (there is no disk $B$ embedded in $S_{(g)}$ whose boundary intersects the diagram in just one crossing). For $g\geq 1$ we can not remove all the crossings that are adjacent to two external regions (Fig.~\ref{figure:reducedD2_g}), that is why we use the word ``simple'' instead of ``reduced''. However most links do not have diagrams with such crossings: if the represented link does not intersect twice any non separating sphere the diagrams with the minimal number of crossings are reduced.

Unfortunately we are not able to get a sharper result than Theorem~\ref{theorem:Tait_conj_g} for $g\geq 3$. Remark~\ref{rem:ineq_psi} shows that if the natural extension of the Tait conjecture for $\Z_2$-homologically trivial links in $\#_g(S^1\times S^2)$ for $g\geq 3$ is true we should prove it with a more complicated method.

The proof for the case of links in $S^3$ and in $S^1\times S^2$ needs just few notions of skein theory. To work in this general case of links in $\#_g(S^1\times S^2)$ we need more complicated tools: either \emph{Turaev's shadows} or \emph{colored trivalent framed graphs}. We think that working with shadows and using the \emph{shadow formula} is the best choice.

\subsection*{Structure of the paper}
In the second section we recall the needed preliminaries about the Jones polynomial, the Kauffman bracket, \emph{skein theory}, diagrams and \emph{shadows}.

In the third section we introduce some further notions and give some propositions.

In the fourth one we prove Theorem~\ref{theorem:Tait_conj_g} and Theorem~\ref{theorem:Tait_conj_Jones_g}. The proof follows the classical one for links in $S^3$ further extending the one of \cite{Carrega_Tait1} for $S^1\times S^2$. 

In the last section we make some conjectures, ask questions and say something that shows that if the natural extension of the Tait conjecture for $\Z_2$-homologically trivial links in $\#_g(S^1\times S^2)$ for $g\geq 3$ is true, we need a more complicated method to prove it.

\subsection*{Acknowledgments}
The author is warmly grateful to Bruno Martelli for his constant support and encouragement.

\section{Preliminaries}\label{section:Preliminaries_g}
In this section we recall some basic notions about the \emph{Jones polynomial}, the \emph{Kauffman bracket}, \emph{skein theory}, \emph{diagrams} and \emph{shadows}. Everything we say is written very briefly and we give some reference.

\subsection{Kauffman bracket}
We can find in \cite{Carrega_Tait1} a more detailed (but still brief) exposition of the topics in this subsection. The \emph{Kauffman bracket} $\langle D \rangle \in \Z[A,A^{-1}]$ was born as a Laurent polynomial associated to a diagram $D$ of a link in $S^3$ or to a framed link. The Kauffman bracket is useful to define and understand the \emph{Jones polynomial} since they differ just by the multiplication of a power of $-A^3$. Thanks to result of Hoste-Przytycki \cite{HP2, Pr1, Pr2} (see also \cite[Proposition 1]{BFK}) and (with different techniques) to Costantino \cite{Costantino2}, it is also defined in the connected sum $\#_g(S^1\times S^2)$ of copies of $S^1 \times S^2$. In the general case of framed links in $\#_g(S^1\times S^2)$ the definition of the Kauffman bracket is based on the notion of ``\emph{skein vector space}'' $K(M)$. This is a vector space over the field of rational functions with rational coefficients $\mathbb{Q}(A)$ that is associated to the oriented 3-manifold $M$. 

The skein space is the vector space with basis the set of framed links of $M$ modulo the sub-space generated by the \emph{skein relations}:
$$
\begin{array}{rcl}
 \pic{1.2}{0.3}{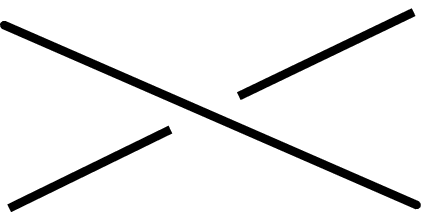}  & = & A \pic{1.2}{0.3}{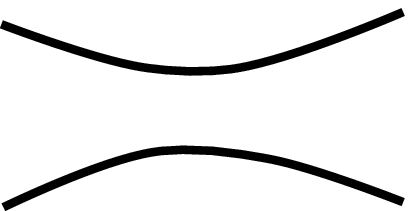}  + A^{-1}  \pic{1.2}{0.3}{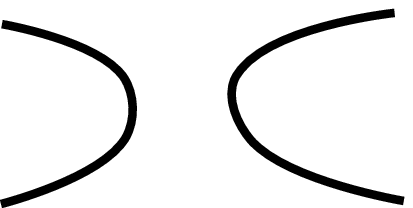}  \\
 L \sqcup \pic{0.8}{0.3}{banp.eps}  & = & (-A^2 - A^{-2})  D  \\
 \pic{0.8}{0.3}{banp.eps}  & = & (-A^2-A^{-2}) \varnothing
\end{array}
$$
These are local relations where the framed links in an equation differ just in the pictured 3-ball that is equipped with a positive trivialization. The elements of a skein space are called \emph{skeins} or \emph{skein elements}. 

The skein vector space of $\#_g(S^1\times S^2)$ is isomorphic tho the base field $\mathbb{Q}(A)$ and is generated by the empty set $\varnothing$. A framed link $L\subset \#_g(S^1\times S^2)$ determines a skein $L\in K(\#_g(S^1\times S^2))$ and as such it is equivalent to $\langle L \rangle \cdot \varnothing$ for a unique coefficient $\langle L \rangle \in \mathbb{Q}(A)$. This coefficient is by definition the \emph{Kauffman bracket} $\langle L \rangle$ of $L$. For $g\leq 1$ it comes out that the Kauffman bracket is a Laurent polynomial (see \cite[Proposition 3.8]{Carrega_Tait1}), but for $g\geq 2$ it is not (see Example~\ref{ex:Kauff}).

\begin{rem}\label{rem:tensor}
There is an obvious canonical linear map $K(M) \to K(M\# N)$ defined by considering a skein in $M$ inside $M\# N$. The linear map $K(\#_g(S^1\times S^2)) \rightarrow K(\#_{g+1}(S^1\times S^2))$ sends $\varnothing$ to $\varnothing$ and hence preserves the bracket $\langle L \rangle$ of $L\subset \#_g(S^1\times S^2)$.

This shows in particular that if $L$ is contained in a 3-ball, the bracket $\langle L \rangle$ is the same that we would obtain by considering $L$ inside $S^3$.
\end{rem}

There is a standard way to color the components of a framed link $L\subset M$ with a natural number and get a skein element of $K(M)$. These colorings are based on particular elements of the \emph{Temperley-Lieb algebra} called \emph{Jones-Wenzl projectors}. Coloring a component with $0$ is equivalent to remove it, while coloring with $1$ corresponds to consider the standard skein. Usually a component colored with $n$ is pictured with a square box with the letter ``$n$'' over the component, or over $n$ parallel copies of the component.

\subsection{Diagrams}

The manifold $\#_g(S^1\times S^2)$ is the double of the 3-dimensional handlebody $H_g$ of genus $g$ (the compact orientable 3-manifold with a handle-decomposition with just $k$ 0-handles and $k+g-1$ 1-handle). We call one such decomposition a \emph{H-decomposition}.

By a theorem of Thom we know that every two embeddings of a 3-disk in a fixed manifold are isotopic. Hence up to isotopies there is a unique Heegaard decoposition of $S^3$ that splits it into two 3-balls.

\begin{theo}\label{theorem:Heegaard_split_S1xS2}
Every two embeddings of the closed surface $\partial H_g$ of genus $g$ in $\#_g(S^1\times S^2)$ that split it into two copies of the the handlebody $H_g$ are isotopic.
\begin{proof}
In \cite[Remark 4.1]{Schultens} it is showed that if we glue two copies of $H_g$ along the boundary to get $\#_g(S^1\times S^2)$, the gluing map must be isotopic to the identity (\cite{Schultens} is an updated and illustrated translation of \cite{Waldhausen}). In \cite[Theorem 1.4]{Carvalho} is shown that two embeddings of $\partial H_g$ that split $\#_g(S^1\times S^2)$ in two copies of $H_g$ and define the identity map $\partial H_g$ (once identified the two embedded handlebodies with $H_g$), are isotopic. For the case $g=1$ we can also see the proof of \cite[Theorem 2.5]{Hatcher}.
\end{proof}
\end{theo}

Since $H_g$ collapses onto a graph, every link in $\#_g(S^1\times S^2)$ can be isotoped in a fixed handlebody of the H-decomposition. The handlebody is the natural 3-dimensional thickening of the disk with $g$ holes $S_{(g)}$. In $H_g$ there are many properly embedded punctured disks $S_{(g)}$ that thicken to $H_g$. If $g=1$, $S_{(g)}$ is an annulus and these embeddings differ by twists. Once one such embedding is fixed $S_{(g)}\rightarrow H_g \subset \#_g(S^1\times S^2)$, every link in $\#_g(S^1\times S^2)$ can be represented by a link diagram in $S_{(g)}$. Of course one such diagram is a generic projection of the link in the embedded disk with $g$ holes.

For $g=1$, $S_{(g)}$ is an annulus. Once a proper embedding of the annulus in a solid torus of the H-decomposition is fixed and given a diagram $D\subset S^1\times [-1,1] = S_{(1)}$ of a link $L\subset S^1\times S^2$, we can get a diagram $D'\subset S^1\times [-1,1]$ that represent $L$ with the embedding of the annulus obtained from the previous one by adding a twist following the move described in Fig.~\ref{figure:twist_diagr}.

\begin{figure}[htbp]
\beq
\picw{4.8}{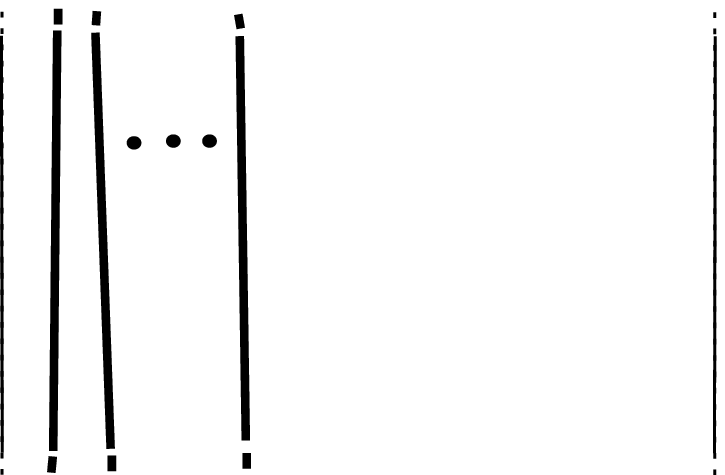} & \longrightarrow & \picw{4.8}{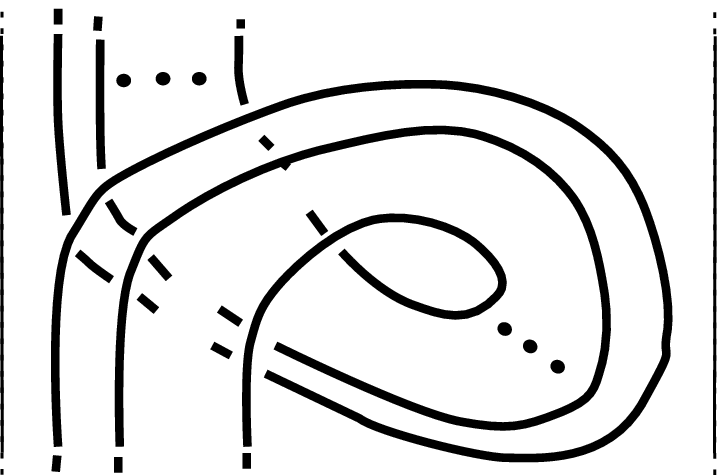} \\
D & & D' 
\eeq
\caption{Two diagrams of the same link in $S^1\times S^2$, the embedding of the annulus $S_{(1)}$ for $D'$ differs from the one of $D$ by the application of a positive twist. The diagrams differ just in the pictured portion that is diffeomorphic to $[-1,1]\times (-1,1)$.}
\label{figure:twist_diagr}
\end{figure}

We note that a link diagram divides $S_{(g)}$ in 2-dimensional connected \emph{regions}. We say that a region of $S_{(g)}$ defined by $D$ is \emph{external} if it touches the boundary of $S_{(g)}$, otherwise it is \emph{internal}.

Every time we will say ``proper embedding of $S_{(g)}$ in a handlebody $H$ of genus $g$'' we will mean also that the handlebody $H$ is a thickening of the embedded punctured disk.

\begin{defn}\label{defn:alt_cr_num}
A link diagram $D\subset S_{(g)}$ is \emph{alternating} if the parametrization of its components $S^1\rightarrow D\subset S_{(g)}$ meets overpasses and underpasses alternately.

Let $L$ be a link in $\#_g(S^1\times S^2)$. The link $L$ is \emph{alternating} if there is an alternating diagram $D\subset S_{(g)}$ that represents $L$ with a proper embedding of $S_{(g)}$ in a handlebody of the H-decomposition of $\#_g(S^1\times S^2)$. 

The \emph{crossing number} of $L\subset \#_g(S^1\times S^2)$ is the minimal number of crossings that a link diagram $D\subset S_{(g)}$ must have to represent $L$ for some proper embedding of $S_{(g)}$ in a handlebody of the H-decomposition of $\#_g(S^1\times S^2)$.
\end{defn}

\begin{rem}
Let $\varphi: \#_g(S^1\times S^2) \rightarrow \#_g(S^1\times S^2)$ be a diffeomorphism and let $L\subset \#_g(S^1\times S^2)$ be a link with a fixed position ($L$ is just a sub-manifold, it is not up to isotopies). Suppose that $L$ is in regular position for a proper embedded disk with $g$ holes $S\subset H_{(1)} \subset \#_g(S^1\times S^2)$ in a handlebody of the H-decomposition $\#_g(S^1\times S^2)= H_{(1)}\cup H_{(2)}$, $H_{(1)} \cong H_{(2)} \cong H_g$. Hence the couple $(L,S)$ defines a link diagram $D\subset S_{(g)}$. Then the link $\varphi(L)$ is in regular position for the punctured disk $\varphi(S)$ that is proper embedded in $\varphi(H_{(1)})$. By Theorem~\ref{theorem:Heegaard_split_S1xS2} $\varphi(H_{(1)}) = H_{(j)}$ (up to isotopy) for some $j=1,2$. The couple $(\varphi(L), \varphi(S))$ still defines the diagram $D\subset S_{(g)}$. Therefore the crossing number and the condition of being alternating are invariant under diffeomorphisms of $\#_g(S^1\times S^2)$.
\end{rem}

This diagrammatic approach is a useful tool to compute the Kauffman bracket.

We say that a link $L$ in a 3-manifold $M$ is $\Z_2$-\emph{homologically trivial} if its homology class with coefficients in $\Z_2$ is null $0=[L]\in H_1(M;\Z_2)$. This is equivalent to say that $L$ bounds an (maybe not orientable) embedded surface in $M$.

\begin{ex}\label{ex:Kauff}
We show in the table below some links in $\#_g(S^1\times S^2)$ together with their Kauffman bracket. They are all $\Z_2$-homologically trivial. In the list there are knots and links with a varying number of components. There are \emph{alternating} and \emph{non alternating} knots and links, Corollary~\ref{cor:conj_Tait_Jones_g} ensures us that example $(5)$ and example $(7)$ are actually non alternating, unfortunately we are not able to say if examples $(6)$ and $(11)$ are not alternating (see Example~\ref{ex:no_alt_g}). Some of them are \emph{H-split} (Definition~\ref{defn:split_homotopic_genus}) and some are not.

In example $(2)$, $\binom g k$ is the binomial coefficient, and the Kauffman bracket can be written as $f/h$, where $f$ and $h$ are the following Laurent polynomials: $f=\sum_{k=0}^g \binom{g}{k} (-A^2-A^{-2})^k$, $h=(-A^2-A^{-2})^{1-g}$. We have $f|_{A^2=i} = 1$, $h|_{A^2=i} =0$, hence $\langle D \rangle$ can not be a Laurent polynomial. 

In example $(6)$ there is a knot whose Kauffman bracket is of the form $1/f$ with $f\in \Z[A,A^{-1}]$ such that $1/f$ is not a Laurent polynomial.

In all the examples except $(8)$, $(10)$ and $(11)$, the Kauffman bracket is of the form $f/\cerchio_1^n$ for some $n\geq 0$ and $f\in \Z[A,A^{-1}]$ ($\cerchio_1= -A^2 -A^{-2}$). Examples $(8)$, $(10)$ and $(11)$ are not of that form. Example $(8)$ and $(11)$ are of the form $f/\cerchio_2$, while $(10)$ is of the form $f/(\cerchio_2\cerchio_3)$ for some $f\in \Z[A,A^{-1}]$ not divided by $\cerchio_2$ or $\cerchio_3$, where $\cerchio_n$ is the Kauffman bracket of the unknot colored with $n$. Note that $\cerchio_n\in \Z[A,A^{-1}]$ and that the roots of $\cerchio_n$ are all roots of unity. Hence these Kauffman brackets have poles in roots of unity different from $q=A^2=i$.

$$
\begin{array}{|c|c|c|}
\hline
 & \text{Diagram} & \langle D \rangle \\
\hline
\hline
(1) & \picw{6.3}{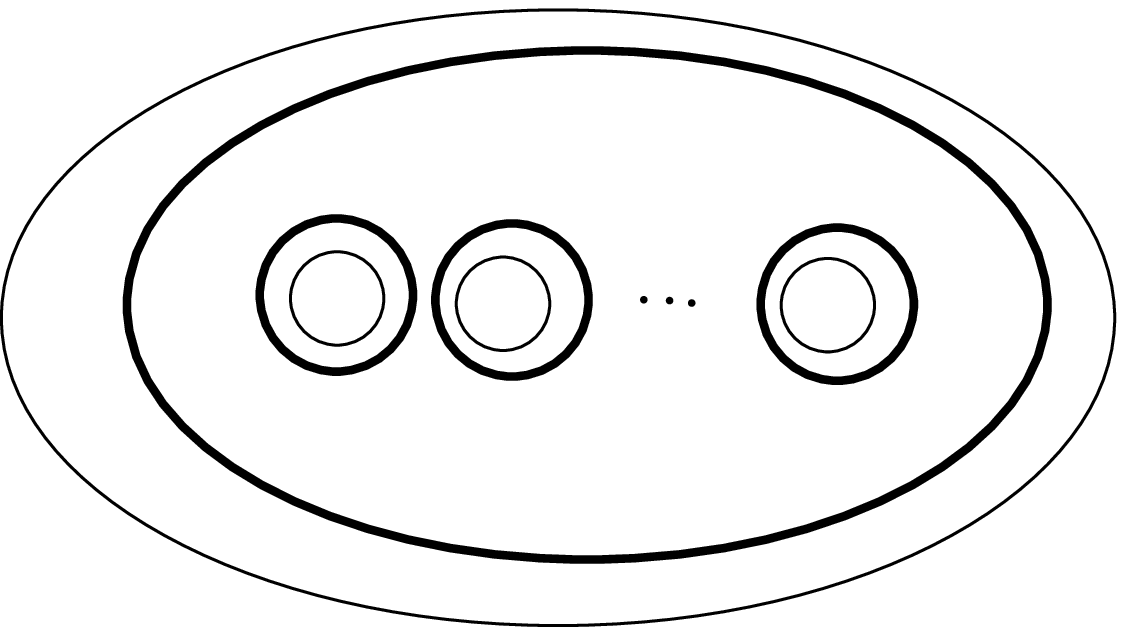} & \parbox[c]{4.5cm}{ 
\begin{center}$
(-A^2-A^{-2})^{1-g} 
$\end{center}
}\\
\hline
(2) & \picw{6.3}{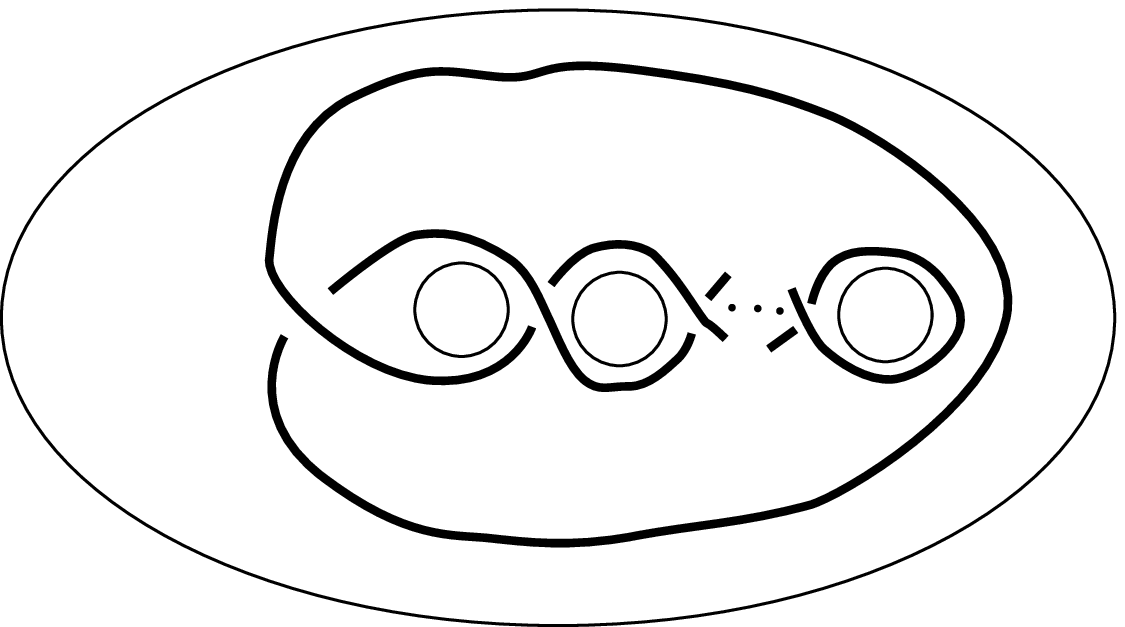} &  \parbox[c]{4.5cm}{
\begin{center}$
\sum_{k=0}^g \binom{g}{k} (-A^2-A^{-2})^{1-g+k} 
$\end{center}
}\\
\hline
(3) & \picw{6.3}{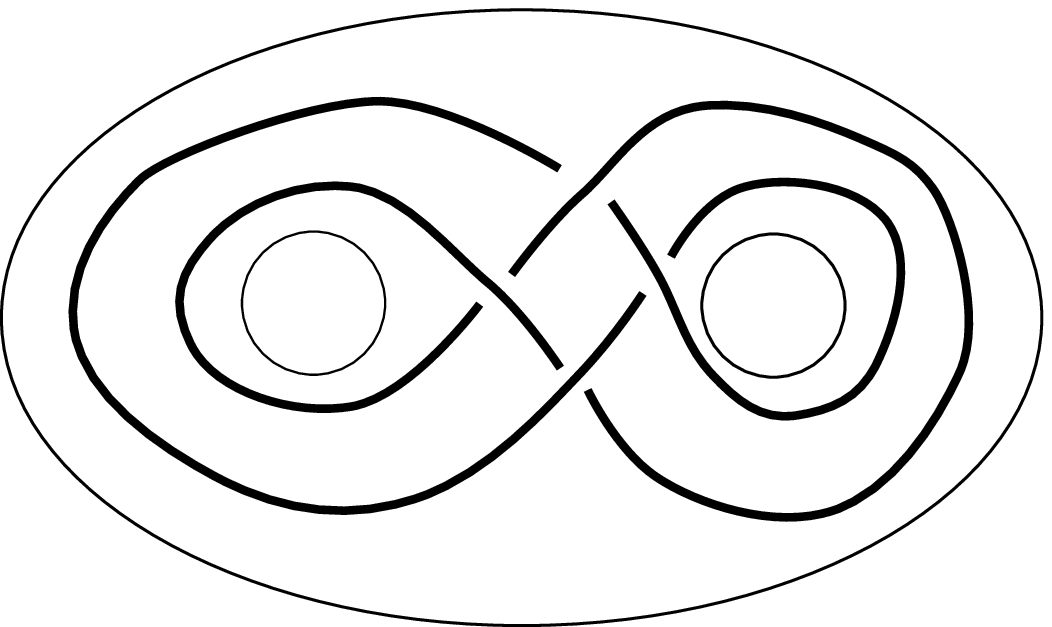} &  \parbox[c]{4.5cm}{
\begin{center}$
\frac{ A^{16} -A^{12} +A^8 +1 }{ A^8 +A^4} 
$\end{center}
}\\
\hline
(4) & \picw{6.3}{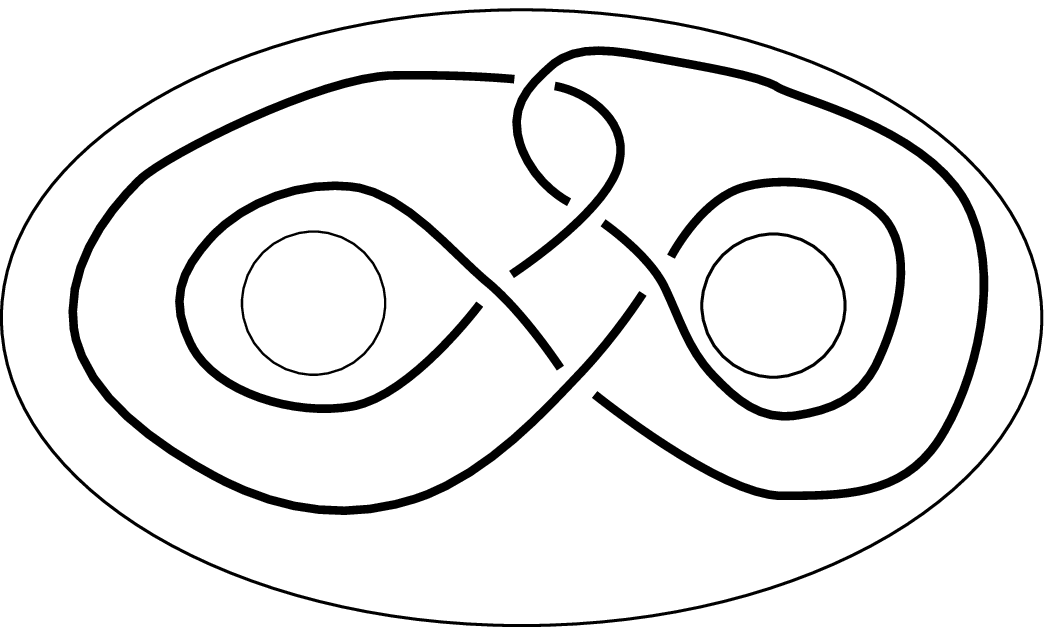} & \parbox[c]{4.5cm}{
\begin{center}$
\frac{ A^{20} -A^{16} +2A^{12} -A^8 +A^4 -1 }{ A^{11} +A^7 }
$\end{center}
}\\
\hline
(5) & \picw{6.3}{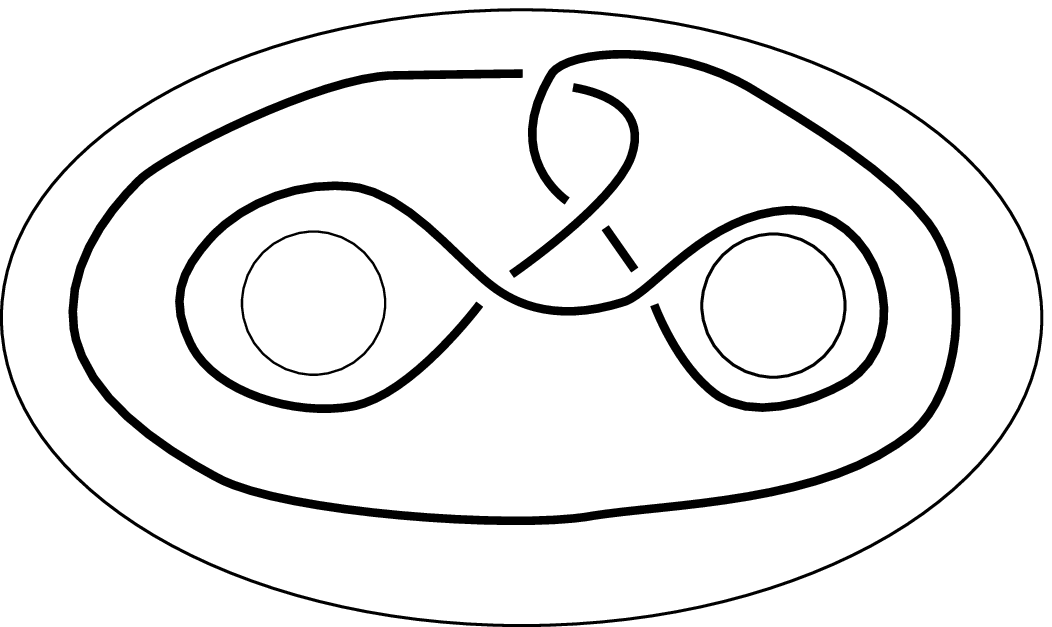} & \parbox[c]{4.5cm}{
\begin{center}$
\frac{ A^{12} -A^{10} +2A^8 +3A^4 +A^2 -1 }{ A^8 +A^4 } 
$\end{center}
}\\
\hline
\end{array}
$$

$$
\begin{array}{|c|c|c|}
\hline
(6) & \picw{6.3}{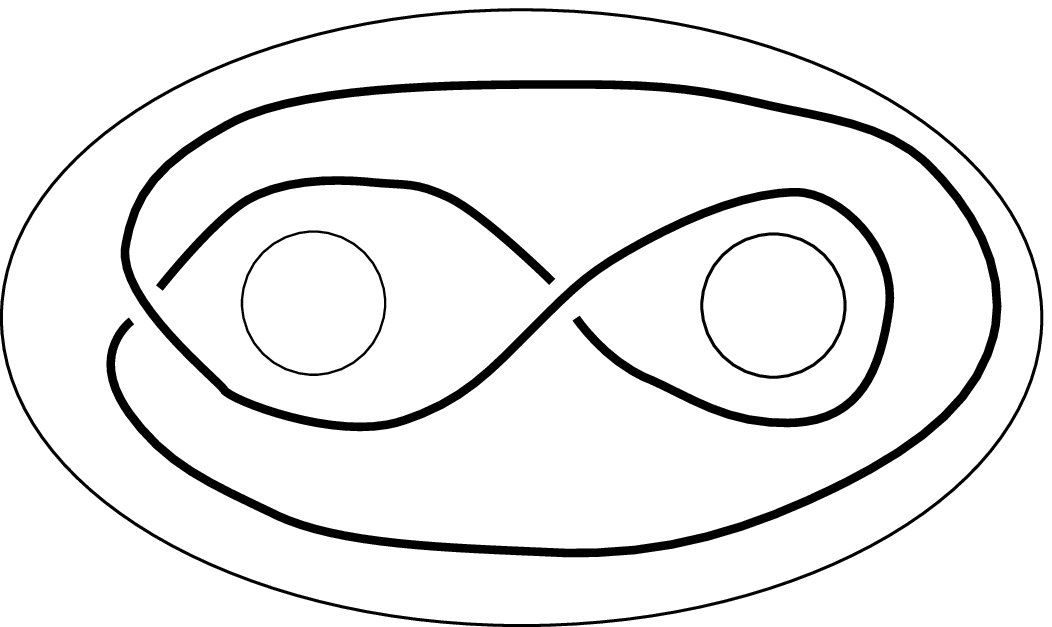} & \parbox[c]{4.5cm}{
\begin{center}$
\frac{1}{ -A^2 -A^{-2} } 
$\end{center}
}\\
\hline
(7) & \picw{6.3}{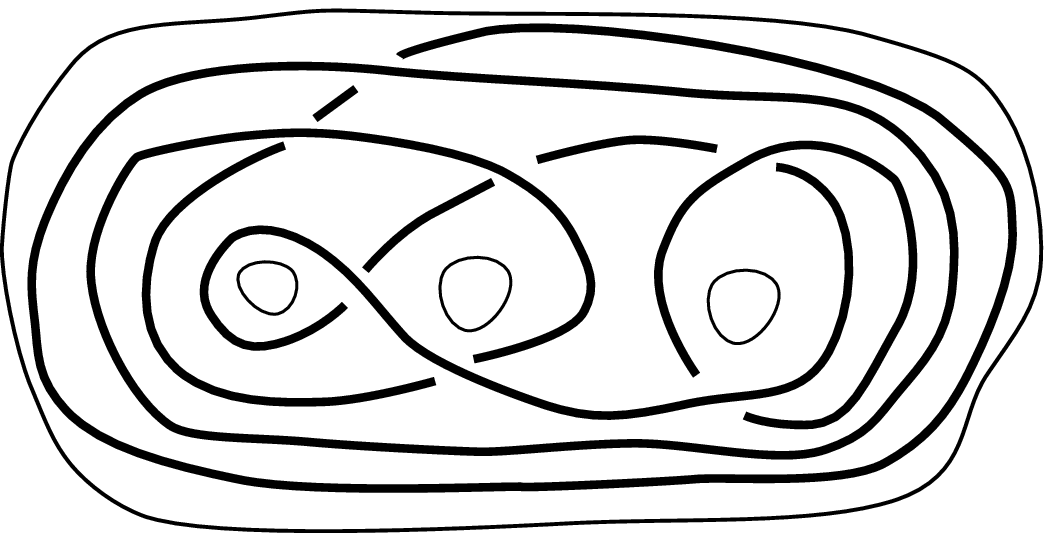} & \parbox[c]{4.5cm}{
\begin{center}$
( -A^{22} -A^{20} -A^{18} +2A^{14} +5A^{12} +3A^{10} +5A^8 +A^6 +4A^4 +1 )/( A^{15} +2A^{11} +A^7) 
$\end{center}
}\\
\hline
(8) & \picw{6.3}{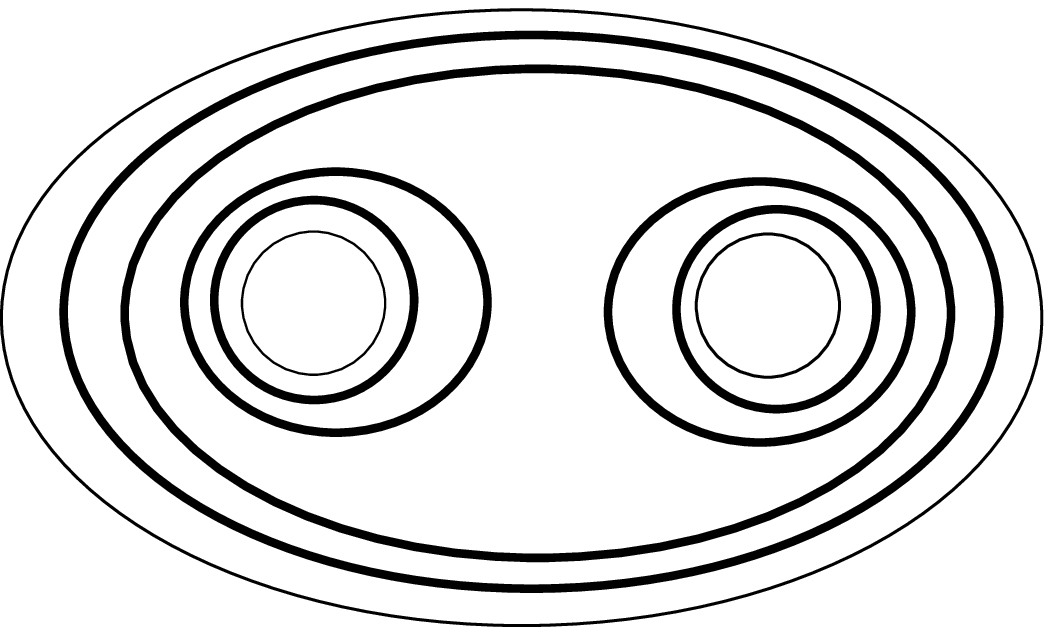} & \parbox[c]{4.5cm}{
\begin{center}$
\frac{ A^8 +2A^4 +1 }{ A^8 +A^4 +1 } 
$\end{center}
}\\
\hline
(9) & \picw{6.3}{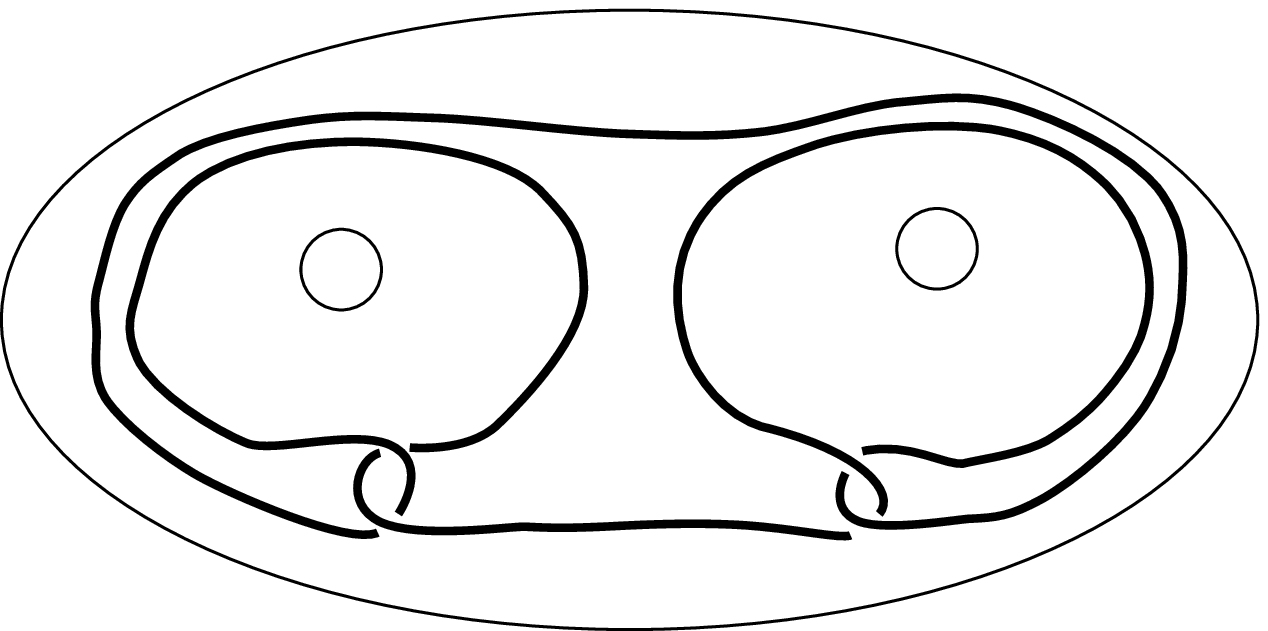} & \parbox[c]{4.5cm}{
\begin{center}$
\frac{-A{16} -2A^8 -1}{ A^{10} +A^6 }
$\end{center}
}\\
\hline
(10) & \picw{6.3}{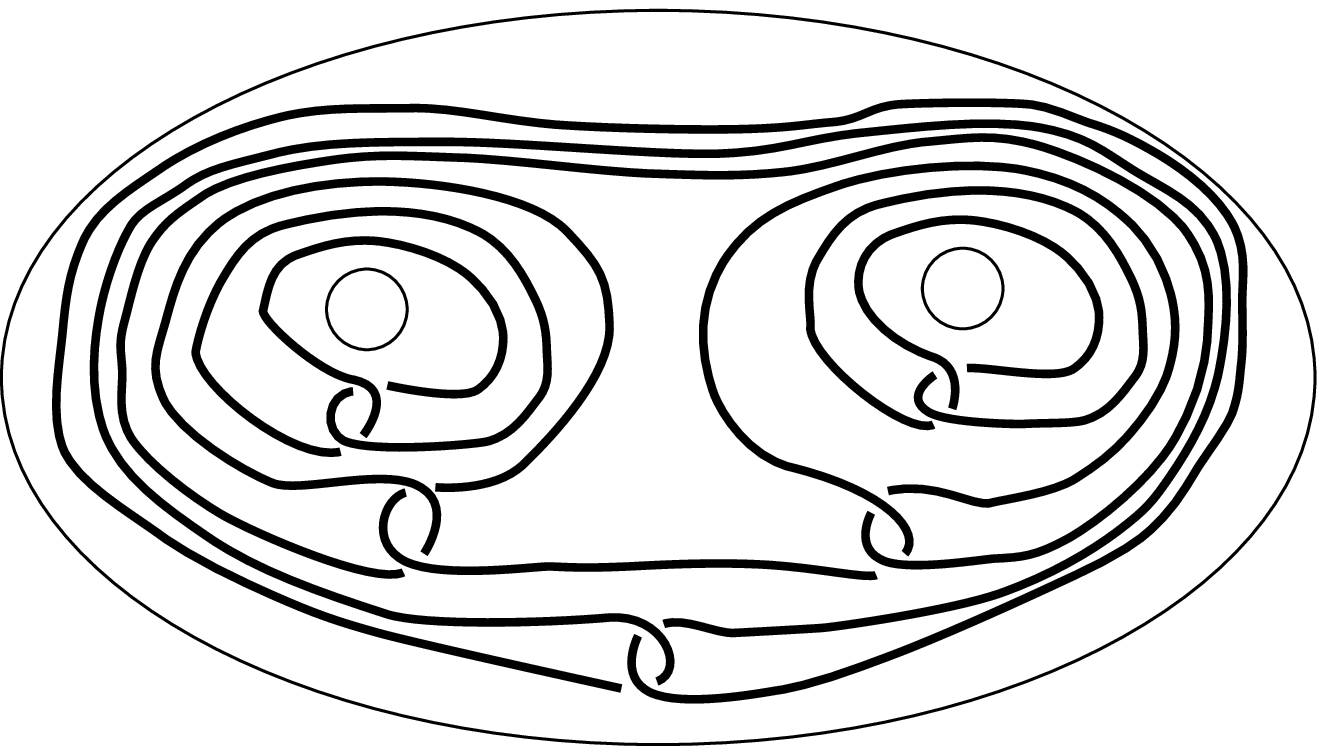} & \parbox[c]{4.5cm}{
\begin{center}$
( 8A^{56} - 15A^{52} + 46A^{48} - 45A^{44} + 106A^{40} -472A^{36} + 97A^{32} - 24A^{28} + 47A^{24} - 14 A^{20} + 22A^{16} - A^{12} + 9A^8 + 2A^4 +1 )/( A^{38} +2A^{34} + 3A^{30} + 3A^{26} + 2A^{22} +A^{18} )
$\end{center}
}\\
\hline
\end{array}
$$

$$
\begin{array}{|c|c|c|}
\hline
(11) & \picw{6.3}{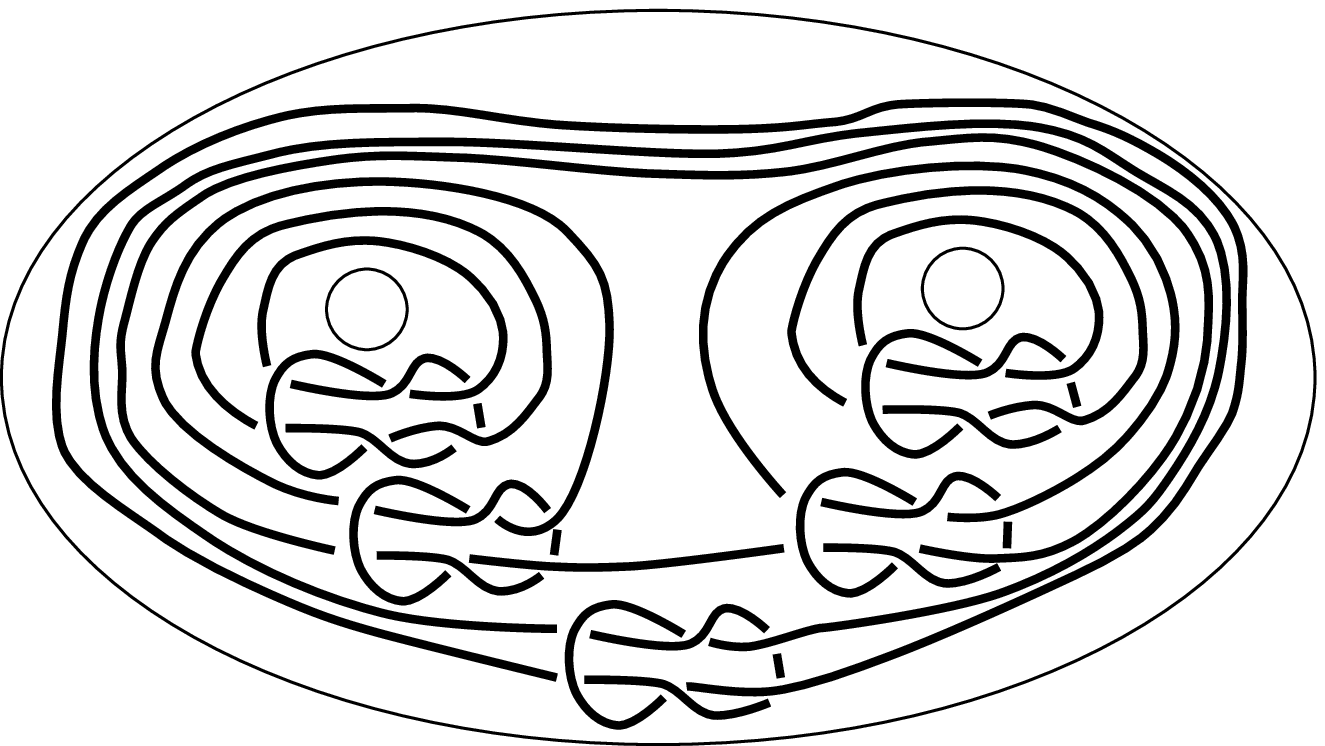} & \parbox[c]{4.5cm}{
\begin{center}$
( -A^{104} + 7A^{100} - 20 A^{96} + 42 A^{92} - 69 A^{88} + 61 A^{84} + 11 A^{80} - 124 A^{76} + 260 A^{72} - 317 A^{68} + 154 A^{64} + 87 A^{60} - 323 A^{56} + 512 A^{52} - 323 A^{48} + 87 A^{44} + 154 A^{40} - 317 A^{36} +260 A^{32} - 124 A^{28} + 112 A^{24} + 61 A^{20} - 69 A^{16} + 42 A^{12} - 20 A^8 + 7A^4 -1 )/( A^{56} + A^{52} + A^{48} )
$\end{center}
}\\
\hline
\end{array}
$$
\end{ex}

\subsection{Shadows}\label{subsec:Shadows}
In this subsection we briefly introduce the notion of ``\emph{shadow}'' and ``\emph{shadow formula}''. They provide us a useful tool to compute quantum invariants. They were introduced by Turaev \cite{Turaev}. We can refer to \cite{Carrega-Martelli, Costantino0, Turaev} for these notions. In particular  we can find a proof of the shadow formula for the Kauffman bracket in \cite[Section 6]{Carrega-Martelli}. 

A \emph{simple polyhedron} $X$ is a connected 2-dimensional compact polyhedron where every point has a neighborhood homeomorphic to one of the five types (1-5) shown in Fig.~\ref{figure:models}. The five types form subsets of $X$ whose connected components are called \emph{vertices} (1), \emph{interior edges} (2), \emph{regions} (3), \emph{boundary edges} (4), and \emph{boundary vertices} (5). The points (4) and (5) altogether form the \emph{boundary} $\partial X$ of $X$. 

\begin{figure}[htbp]
\begin{center}
\includegraphics[width = 12 cm]{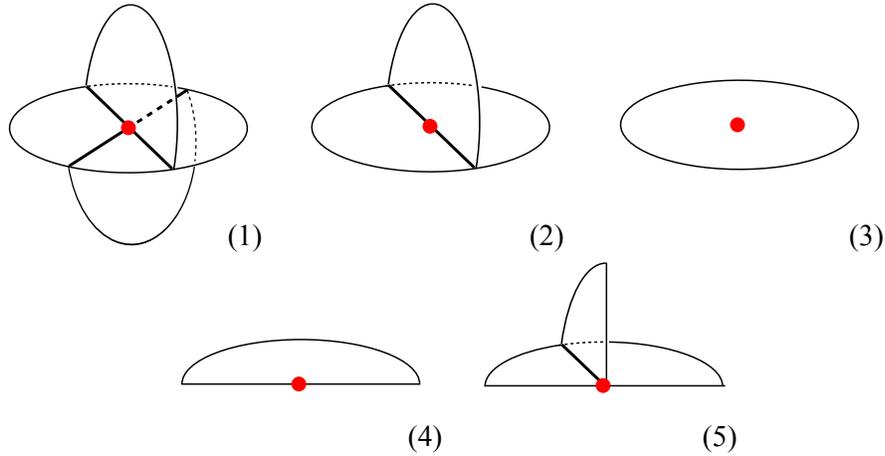}
\caption{Neighborhoods of points in a simple polyhedron.}
\label{figure:models}
\end{center}
\end{figure}

A \emph{shadow} of an oriented compact 4-manifold $W$ is a simple polyhedron $X$ properly embedded in $W$, $X\subset W$, $X \cap \partial W = \partial X$, such that it is locally flat and $W$ collapses onto $X$. The boundary $\partial X$ is equipped with a surface in $\partial W$ that collapses on it, we call this surface \emph{framing}. Let $L$ be a framed link in an oriented closed 3-manifold $M$. A \emph{shadow} $X$ of $(L,M)$ is a shadow of a 4-manifold $W$ such that $\partial W = M$ and $\partial X= L$. Every link $L	\subset \#_g(S^1\times S^2)$ has a shadow that collapses onto a graph.

The embedding of a shadow in the 4-manifold provides a half-integer ${\rm gl}(R)\in \frac 1 2 \Z$ to each region $R$ of $X$ called \emph{gleam}. Gleams generalize the Euler number of $D^2$-bundles over surfaces. A shadow is completely determined (up to diffeomorphism of the 4-manifold) by its topology and its gleams. 

We can get all quantum invariants from shadows with what are called \emph{shadow formulas}. They look like Euler characteristics: they are composed by elementary bricks associated to the maximal connected pieces of dimension $0$ (vertices), $1$ (edges) and $2$ (regions) of the shadow, and they are combined together with a ``sign'' depending on the parity of the dimension. 

A triple $(a,b,c)$ of non negative integers is \emph{admissible} if they satisfy the triangular inequalities: $a\leq b+c$, $b\leq c+a$, $c\leq a+b$, and the sum is even $a+b+c\in 2\Z$. An \emph{admissible coloring} $\xi$ of a shadow $X$ is the assignment of a non negative integer $\xi(R)$ (a color) to each region $R$ of $X$ such that for each edge $e$ the colors of the three regions adjacent to $e$ form an admissible triple. If the boundary $\partial X$ has a coloring (a non negative integer associated to each edge of $X$), we say that a coloring $\xi$ of $X$ extends the boundary coloring if for every region $R$ adjacent to the boundary edge $e_\partial$ the color $\xi(R)$ is equal to the color of $e_\partial$. If the set ${\rm Adm}(X)$ of the admissible colorings of $X$ that extend the boundary coloring is non empty, we have that ${\rm Adm}(X)$ is finite if and only if $X$ collapses onto a graph.

Let $X\subset W$ be a shadow that collapses onto a graph of genus $g$ whose boundary has no vertices, namely it is a framed link $L=\partial X$. The ambient oriented 4-manifold $W$ must collapse onto the graph, hence it must be the 4-dimensional handlebody of genus $g$ (the oriented compact 4-manifold with a handle-decomposition with just $k>0$ 0-handles and $h-g$ 1-handles), therefore $\partial W= \#_g(S^1\times S^2)$. Suppose that $L$ has a fixed coloring, for instance all the components are colored with $1$. The \emph{shadow formula} says that the Kauffman bracket of $L\subset \#_g(S^1\times S^2)$ is equal to
$$
\langle L \rangle = \sum_\xi \frac{\prod_f \cerchio_f^{\chi(f)}A_f \prod_v \tetra_v }
{\prod_e \teta_e^{\chi(e)} }  .
$$
Here the sum is taken over all the admissible colorings of $X$ that extend the boundary coloring and the product is taken on all regions $f$, inner edges $e$, inner vertices $v$. The symbols
$$
\cerchio_f,\ \teta_e, \  \tetra_v
$$
denote the skein element of these graphs in $K(S^3)=\mathbb{Q}(A)$, colored as $f$, or as the regions incident to $e$, or $v$. The \emph{phase} $A_f$ is a value depending on the color $\xi(f)$ and the gleam $\gl(f)$ of the region $f$:
$$
A_f:= (-1)^{\gl(f)\xi(f)} A^{-\gl(f)\xi(f)( \xi(f) +2 )} ,
$$
Note that if the gleam of $f$ is null $A_f=1$.

We are going to use diagrams, skein relations and the shadow formula to easily compute the Kauffman bracket of links in $\#_g(S^1\times S^2)$.

\section{The extended Tait conjecture}

In this section we provide more notions and results in order to state and prove the main theorems.

In $S^3$ all link diagrams with the minimal number of crossings are reduced. This statement has been discussed in \cite{Carrega_Tait1} for links in $S^1\times S^2$. Clearly the diagrams in $S_{(g)}$ with the minimal number of crossings do not have crossings as the ones of Fig.~\ref{figure:reducedD_g} (there is no disk $B$ embedded in $S_{(g)}$ whose boundary intersects the diagram in just one crossing). For $g\geq 1$ we can not remove all the crossings that are adjacent to two external regions (Fig.~\ref{figure:reducedD2_g}). For $g=1$ we have the definition of \emph{quasi-simple} diagram that is obtained relaxing the definition of simple admitting at most one crossing adjacent to two external regions (Fig.~\ref{figure:reducedD2_g}). We have that every diagram in the annulus of a link in $S^1\times S^2$ that has the minimal number of crossings is quasi-simple. Unfortunately we can not get the analogous result for $g\geq 2$ with a so easy definition of quasi-simple. However most links do not have diagrams with crossings adjacent to two external regions: if the represented link does not intersect twice any non separating sphere the diagrams with the minimal number of crossings are simple.

A dichotomy between links in $S^1\times S^2$ is shown in \cite{Carrega_Tait1}: if the link is $\Z_2$-homologically trivial the proper version of the Tait conjecture holds, otherwise it does not. Following the method of \cite[Subsection 3.1]{Carrega_Tait1} we can show that the diagrams in Fig.~\ref{figure:no_Tait_g} represent the same knot $K\subset \#_g(S^1\times S^2)$ once a proper embedding of $S_{(g)}$ in a handlebody of the H-decomposition is fixed. The knot $K$ is $\Z_2$-homologically non trivial and the diagrams are both alternating and simple, but they have a different number of crossings. Therefore even in the general case of links in $\#_g(S^1\times S^2)$ the natural extension of the Tait conjecture does not hold for $\Z_2$-homologically non trivial links. 

We can start to see this dichotomy given by the $\Z_2$-homology class also by looking at the Kauffman bracket. In fact \cite[Proposition 3.9]{Carrega_Tait1} shows that if a link in $S^1\times S^2$ is $\Z_2$-homologically non trivial its Kauffman bracket is null. To prove that we just needed three well known relations in skein theory (two of them are pictured in Fig.~\ref{figure:sphere}). We have the analogous result in $\#_g(S^1\times S^2)$, but to get it we need more complicated tools (shadows or trivalent colored graphs):
\begin{prop}\label{prop:0Kauff_g}
Let $L\subset \#_g(S^1\times S^2)$ be a framed link. Suppose that $L$ is $\Z_2$-homologically non trivial
$$
[L] \neq 0 \in H_1(\#_g(S^1\times S^2); \Z_2) .
$$
Then
$$
\langle L \rangle = 0 .
$$
\begin{proof}
Let $X$ be a shadow of $(L, \#_g(S^1\times S^2))$ collapsing onto a graph. The 4-dimensional thickening of $X$ is the 4-dimensional handlebody of genus $g$, $W_g$. There is a graph $\Gamma \subset W_g$ such that $W_g\setminus \Gamma$ is a collar of the boundary, namely it is diffeomorphic to $\#_g(S^1\times S^2) \times [0,1) $. The homology class of $L$ in $H_1(W_g ;\Z_2)$ is $0$ if and only if $L$ bounds a surface $S\subset W_g$. By transversality we can suppose that $S$ does not intersect $\Gamma$. Hence $L$ is $\Z_2$-homologically trivial in $W_g$ if and only if it is so in $\#_g(S^1\times S^2) \times [0,1)$, thus in $\#_g(S^1\times S^2)$. 

Given an admissible coloring $\xi$ of $X$ that extends the one of $L$ (every component is colored with $1$), the regions having odd colors form a surface $S_\xi \subset W_g$ bounded by $L$.

 By hypothesis $L$ is $\Z_2$-homologically non trivial in $\#_g(S^1\times S^2)$. Hence for what said above a surface like $S_\xi$ can not exists. Therefore there are no admissible colorings of $X$ that extend the one of $L$. Hence by the shadow formula $\langle L \rangle = 0$.
\end{proof}
\end{prop}

\begin{figure}
\begin{center}
\includegraphics[width = 10 cm]{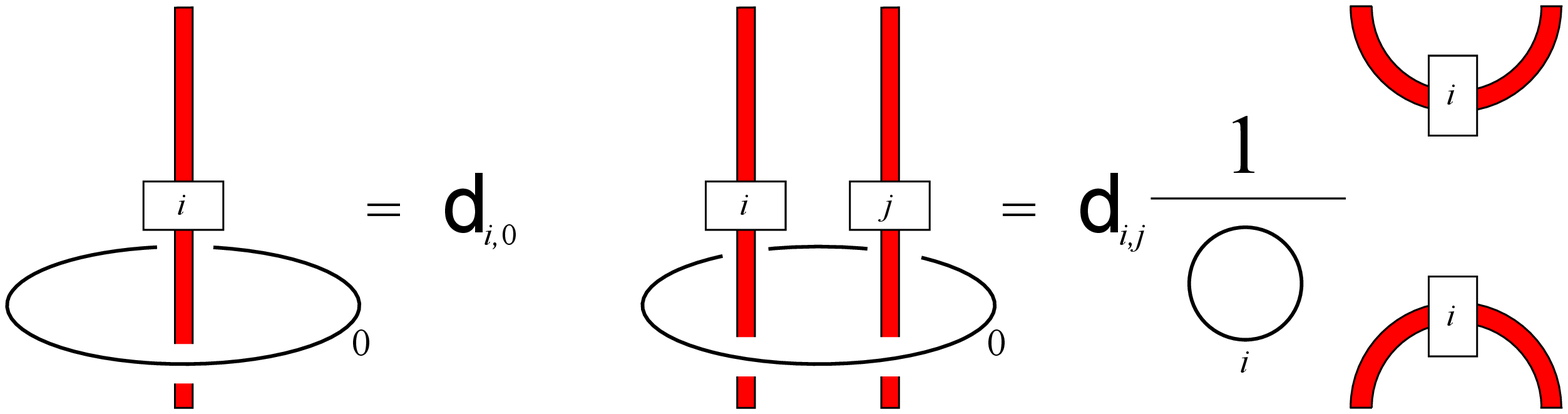}
\caption{Sphere intersection. The symbol $d_{i,j}$ is the Kronecker delta.}
\label{figure:sphere}
\end{center}
\end{figure}

We recall the notion of ``\emph{state}'' and the following notations:
\begin{defn}
Let $D$ be a link diagram in $S_{(g)}$. A \emph{Kauffman state} of $D$, or simply a \emph{state}, is a function $s$ from the set of crossings of $D$ to $\{1,-1\}$. The assignment of $\pm 1$ to a crossing determines a unique way to remove that crossing as described in Fig.~\ref{figure:splitting}. Hence a state removes all the crossings producing a finite collection of non intersecting circles in the surface. This collection  of circles is called the \emph{splitting}, or the \emph{resolution}, of $D$ with $s$. We denote with 
\begin{itemize}
\item{$s_+$ and $s_-$ the constant states that assign respectively always $+1$ and always $-1$;}
\item{$sD$ the number of homotopically trivial components of the splitting of $D$ with the state $s$;}
\item{$\sum s(i)$ for the sum over the crossings of the signs assigned by the state $s$;}
\item{$D_s$ for the diagram obtained removing the homotopically trivial components from the splitting of $D$ with $s$;}
\item{$g(D)$ for the minimal number of holes of a punctured disk containing $D$ and contained in $S_{(g)}$.}
\end{itemize}
\end{defn}

\begin{figure}[htbp]
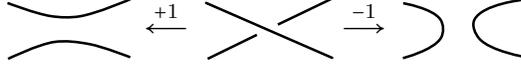

$$
\pic{1.6}{0.4}{Acanalep.eps} \ \stackrel{+1}{\longleftarrow} \ \pic{1.6}{0.4}{incrociop.eps} \ \stackrel{-1}{\longrightarrow} \ \pic{1.6}{0.4}{Bcanalep.eps}
$$
\caption{The splitting of a crossing.}
\label{figure:splitting}
\end{figure}

We have the usual notion of ``\emph{adequacy}'':
\begin{defn}
A link diagram $D\subset S_{(g)}$ is \emph{plus-adequate} (resp. \emph{minus-adequate}) if $s_+D > s_+D'$ ($s_-D > s_-D'$) for every diagram $D'$ obtained changing the over/underpass of one crossing.
\end{defn}

\begin{prop}\label{prop:reduced_D_g}
Fix an embedding of $S_{(g)}$ in a handlebody of the H-decomposition of $\#_g(S^1\times S^2)$. Every simple, alternating, connected link diagram $D\subset S_{(g)}$ that represents a $\Z_2$-homologically trivial link in $\#_g(S^1\times S^2)$ is adequate.
\begin{proof}
The proof of the case $g=1$ \cite[Proposition 3.14]{Carrega_Tait1} applies as is also here. 
\end{proof}
\end{prop}

\begin{defn}\label{defn:split_homotopic_genus}
A link $L\subset \#_g(S^1\times S^2)$ is \emph{H-split} if there is a non connected diagram $D\subset S_{(g)}$ that represent $L$ by some embedding of $S_{(g)}$ in a handlebody of the H-decomposition of $\#_g(S^1\times S^2)$ (non connected as a planar graph).

The \emph{homotopic genus} of $L$ is the minimum of $g(D)$ as $D$ varies among all diagrams of $L$ in any embedded disk with $g$ holes.
\end{defn}

\begin{prop}\label{prop:homot_genus}
The homotopic genus of a link $L\subset \#_g(S^1\times S^2)$ is equal to the minimum $g'$ such that the complement has a connected sum decomposition of the form 
$$
\#_g(S^1\times S^2)\setminus L = (\#_{g'}(S^1\times S^2)\setminus L') \# (\#_{g-g'}(S^1\times S^2))
$$
for some link $L'\subset \#_{g'}(S^1\times S^2) $.
\begin{proof}
If $L$ has homotopic genus $g'\leq g$ there is a diagram $D\subset S_{(g)}$ that represents $L$ via some embedding of $S_{(g)}$ that is contained in a disk with $g'$ holes lying in $S_{(g)}$. Hence there is a factor $\#_{g-g'}(S^1\times S^2)$ in a connected sum decomposition of the complement of $L$. On the other hand if we have a decomposition $\#_g(S^1\times S^2)= (\#_{g'}(S^1\times S^2)\setminus L') \# (\#_{g-g'}(S^1\times S^2))$, we can get a diagram $D\subset S_{(g)}$ representing $L$ via some embedding of $S_{(g)}$ by adding $g-g'$ 1-handles to $S_{g'}$ equipped with a diagram $D'\subset S_{(g')}$ of $L'$. This implies that the homotopic genus of $L$ is at most $g'$.
\end{proof}
\end{prop}

\begin{rem}
By Proposition~\ref{prop:homot_genus} if a link $L\subset \#_g(S^1\times S^2)$ has homotopic genus $g'<g$, its complement $\#_g(S^1\times S^2)\setminus L$ is reducible, in particular it is not hyperbolic. By Proposition~\ref{prop:homot_genus} and Remark~\ref{rem:tensor} the links with homotopic genus $g'\leq g$ can be seen as links in $\#_{g'}(S^1\times S^2)$, hence the really interesting links in $\#_g(S^1\times S^2)$ are the ones with homotopic genus $g$.
\end{rem}

We need to extend the notion of ``\emph{breadth}'' from Laurent polynomials to rational functions:
\begin{defn}\label{defn:breadth_g}
Let $f = g/h$ be a rational function and $g$ and $h$ be two polynomials. The \emph{degree}, or the \emph{order} in $\infty$, of $f$ is the difference of the maximal degree of the non zero monomials of $g$, $\ord_\infty g$, and the maximal degree of the non zero monomials of $h$, $\ord_\infty h$:
$$
\ord_\infty f := \ord_\infty g -\ord_\infty h .
$$ 
The \emph{order} in $0$ of $f$ is the difference of the minimal degree of the non zero monomials of $g$, $\ord_\infty g$, and the minimal degree of the non zero monomials of $h$, $\ord_0 h$:
$$
\ord_0 f := \ord_0 g -\ord_0 h.
$$
The \emph{breadth} of $f$ is 
$$
B(f):= \ord_\infty f - \ord_0 f = B(g) -B(h).
$$
If $f=0$ we have $\ord_\infty f := -\infty$ and $\ord_0 f:= \infty$ and we define $B(f):=0$.
\end{defn}
Clearly the definitions above do not depend on the choice of $g$ and $h$. If $0\leq \ord_0 f < \infty$, $\ord_0 f$ is equal to the multiplicity of $f$ in $0$ as a zero. If $\ord_0 f \leq 0$, $\ord_0 f$ is the opposite of the order of $f$ in $0$ as a pole. If $-\infty < \ord_\infty f \leq 0$, $\ord_\infty f$ is the opposite of the multiplicity of $f$ in $\infty$ as a zero. If $0 \leq \ord_\infty f$, $\ord_\infty f$ is the order of $f$ in $\infty$ as a pole. Here are some easy and useful properties:
\begin{itemize}
\item{if $f,g\neq 0$ then $B(f/g) = B(f)-B(g)$;}
\item{$\ord_0 f(A)= - \ord_\infty f(A^{-1})$;}
\item{if $f$ is \emph{symmetric} (for all $A$ we have $f(A)=f(A^{-1})$), then $\ord_\infty f = -\ord_0 f$;}
\item{$\ord_\infty (f+g) \leq \max\{ \ord_\infty f , \ord_\infty g\}$;}
\item{$\ord_\infty (f+g) = \max\{ \ord_\infty f , \ord_\infty g\}$ if $\ord_\infty f \neq \ord_\infty g$;}
\item{$\ord_0 (f+g) \geq \min\{ \ord_0 f , \ord_0 g\}$;}
\item{$\ord_0 (f+g) = \min\{ \ord_0 f , \ord_0 g\}$ if $\ord_0 f \neq \ord_0 g$;}
\item{$\ord_\infty (f\cdot g) = \ord_\infty f +\ord_\infty g$, $\ord_0 (f\cdot g) = \ord_0 f +\ord_0 g$;}
\item{$\ord_\infty \frac 1 f = -\ord_\infty f$, $\ord_0 \frac 1 f = -\ord_0 f$.}
\end{itemize}
Let $\cerchio_n$ be the Kauffman bracket of a 0-framed homotopically trivial unknot colored with $n$. We have
$$
\cerchio_n = (-1)^n [n+1] , \ \ [n+1] = \frac{A^{2(n+1)} - A^{-2(n+1)} }{ A^2 -A^{-2} } .
$$
We note that
$$
\ord_\infty \cerchio_n = -\ord_0 \cerchio_n = 2n .
$$

Let $D\subset S_{(g)}$ be a link diagram. Using the skein relations we get
$$
\langle D \rangle = \sum_s \langle D |s \rangle ,
$$
where
$$
\langle D| s\rangle := A^{\sum s(i)} (-A^2-A^{-2})^{sD} \langle D_s\rangle .
$$
If the splitting of $D$ with $s$ has only homotopically trivial components, $D_s$ is empty and $\langle D_s\rangle =1$. We can easily get a shadow $X_s$ from $D_s$ attaching to $S_{(g)}$ an annulus along each component of $D_s$ and giving to each region gleam $0$. The boundary $\partial X_s$ of $X_s$ consists of the components of $D_s$ plus the $g+1$ boundary components of $S_{(g)}$. We color the boundary components of $X_s$ corresponding to the components of $D_s$ with $1$ (the remaining boundary components of the annuli attached to the components of $D_s \subset S_{(g)}$), and the other ones with $0$. The polyhedron $X_s$ is a shadow of a framed link in $\#_g(S^1\times S^2)$ that is the union of the link described by $D_s$ and a link colored with $0$. Therefore the shadow formula applied to $X_s$ gives $\langle D_s\rangle$. Hence
$$
\langle D_s \rangle = \sum_\xi \prod_R \cerchio^{\chi(R)}_{\xi(R)} ,
$$
where $\xi$ runs over all the admissible colorings of $X_s$ that extend the coloring of the boundary, $R$ runs over all the regions of $X_s$ (the external regions do not matter because either they are annuli or their color is $0$), $\chi(R)$ is the Euler characteristic of $R$, and $\xi(R)$ is the color of $R$ given by $\xi$. Therefore $\langle D_s \rangle$ is a symmetric function of $A^2$, namely there are two polynomials $f,h\in \Z[q]$ such that 
$$
\langle D_s \rangle = \frac{f|_{q=A^2}}{h|_{q=A^2}}  , \ \ \langle D_s \rangle|_A = \langle D_s \rangle|_{A^{-1}}  .
$$

We define:
\begin{defn}\label{defn:psi}
Let $s$ be a state of the link diagram $D\subset S_{(g)}$. We set
$$
\psi(s):= \frac 1 2 \ord_\infty \langle D_s \rangle = -\frac 1 2 \ord_0 \langle D_s \rangle \ \ \in \Z .
$$
\end{defn}

In particular we get the following:
\begin{prop}
Let $D\subset S_{(g)}$ be a $n$-crossing link diagram. Then there are two polynomials $f,h\in \Z[q]$ such that
$$
\langle D \rangle = \begin{cases}
\frac{f|_{q=A^2}}{h|_{q=A^2}} & \text{if } n \in 2\Z \\
A\cdot \frac{f|_{q=A^2}}{h|_{q=A^2}} & \text{if } n \in 2\Z+1 \\
\end{cases} .
$$
\end{prop}

\section{Proof of the theorems}

In this section we prove Theorem~\ref{theorem:Tait_conj_g} and Theorem~\ref{theorem:Tait_conj_Jones_g}. As in \cite{Carrega_Tait1} we follow the classical proof of Kauffman, K. Murasugi and Thistlethwaite (for instance we can also see \cite[Chapter 5]{Lickorish}) applying some modifications for our case.

The following lemma is extremely useful to manage the quantity $\psi(s)$ and is needed specifically for the $g \geq 2$ case: no similar result is contained in \cite{Carrega_Tait1}.

\begin{lem}\label{lem:psi}
Let $D\subset S_{(g)}$ be a diagram of a $\Z_2$-homologically trivial link in $\#_g(S^1\times S^2)$ and let $s$ be a state of $D$ such that $D_s$ is non empty. Then the following hold:
\begin{enumerate}
\item{there is a unique admissible coloring $\xi_0$ of the shadow $X_s$ that extends the boundary coloring and assigns only the colors $0$ and $1$;}
\item{$$
\psi(s) = \max_\xi \frac 1 2 \ord_\infty \prod_R \cerchio^{\chi(R)}_{\xi(R)} = \max_\xi \sum_R \chi(R) \cdot \xi(R) ;
$$}
\item{$$
\psi(s)= \sum_R \chi(R) \cdot \xi_0(R) .
$$}
\end{enumerate}
\begin{proof}
Every meromorphic function $f:\mathbb{C} \dashrightarrow \mathbb{C}$ is of the form $f(A) =\lambda_f A^{\ord_\infty f} + g(A)$, where $\lambda_f$ is a unique non zero complex number and $g:\mathbb{C} \dashrightarrow \mathbb{C}$ is a unique meromorphic function such that $\ord_\infty g < \ord_\infty f$. The equality $\ord_\infty (f_1+f_2) = \max\{ \ord_\infty f_1 , \ord_\infty f_2\}$ holds if and only if either $\ord_\infty f_1 \neq \ord_\infty f_2$ or $\lambda_{f_1} \neq - \lambda_{f_2}$. For an admissible coloring $\xi$ of the shadow $X_s$ let $f(\xi)$ be its contribution to the shadow formula:  
$$
f(\xi) := \prod_R \cerchio^{\chi(R)}_{\xi(R)} .
$$
We get the second statement by showing that $\lambda_{f(\xi)} = \pm 1$ and the sign does not depend on $\xi$ for each coloring $\xi$. We have 
$$
f(\xi) = (-1)^{\sum_R \chi(R)\xi(R)} \prod_R [\xi(R) +1]^{\chi(R)} ,
$$
where 
$$
[n+1] = \frac{A^{2(n+1)} - A^{-2(n +1)} }{ A^2 - A^{-2} } .
$$

The regions given by a diagram are colored odd or even like in a chessboard. Therefore the parity of the colors of the regions does not depend on the choice of the admissible coloring. Hence $(-1)^{\sum_R \chi(R)\xi(R)}$ does not depend on $\xi$. We have $\lambda_{[n+1]} = 1$. Therefore $(-1)^{\sum_R \chi(R)\xi(R)} \lambda_{f(\xi)} = 1$.

Since the link represented by $D$ is $\Z_2$-homologically trivial and the first skein relation does not change the $\Z_2$-homology class, we have that the link $L_s\subset \#_g(S^1\times S^2)$ represented by $D_s$ is $\Z_2$-homologically trivial. The link $L_s$ is contained in the punctured disk $S_{(g)}$ that is contained in the handlebody $H_g$ whose double is $\#_g(S^1\times S^2)$. Hence $0 = [L_s]= [D_s] \in H_1(S_{(g)};\Z_2)$. Therefore we can find a surface $S_s \subset S_{(g)}$ merging some internal regions of $S_{(g)}$ (given by $D_s$) and whose boundary is $D_s$. We color with $1$ the regions that compose $S_s$ and the annuli attached to the components of $D_s$, and we color with $0$ the remaining ones. We found an admissible coloring of $X_s$ that extends the boundary coloring and assigns just color $0$ and $1$. Such coloring is unique because, as we observed before, the parity of the color of a region does not depend on the admissible coloring. Let $\xi_0$ be such coloring. Since the Euler characteristics of the regions are all non positive and the colorings are all non negative, we have $\sum_R \chi(R)\xi(R) \leq \sum_R \chi(R)\xi_0(R)$ for every admissible coloring $\xi$. Therefore using this and the second statement we get the third statement.
\end{proof}
\end{lem}

Although we do not need it to prove the main theorems, it is interesting to get also a lower bound of the quantity $\psi(s)$:
\begin{prop}
Let $D\subset S_{(g)}$ be a link diagram and $s$ a state such that $D_s$ is non empty. For $1\leq h \leq g$ let $\varphi_h(D_s)$ be the number of internal regions given by $D_s$ that are disks with $h$ holes. Then
\begin{enumerate}
\item{
$$
\sum_{h=2}^{g(D)} (h-1)\varphi_h(D_s) \leq g(D_s) -1 ;
$$}
\item{
$$
1-g(D) \leq 1-g(D_s) \leq \psi(s) .
$$}
\end{enumerate}
\begin{proof}
$(1)$ We proceed by induction on $g(D)$. If $D$ is contained in an annulus ($g(D)=1$) obviously the equality holds. Suppose that the inequality is true for every diagram $D'$ such that $g(D') < g(D)$. There are two cases:
\begin{enumerate}
\item{there is an internal region that is a disk with $g$ holes ($\varphi_g(D)=1$) and the other internal regions are annuli ($\varphi_h(D)=0$ for $1<h<g$);}
\item{there are no internal regions that are disks with $g$ holes ($\varphi_g(D)=0$).}
\end{enumerate}
In the first case the equality holds. Suppose we are in the second case. The diagram $D$ is obtained merging three disjoint diagrams, $D_1, D_2 , D_3 \subset S_{(g)}$, such that $g(D_1) + g(D_2) =g(D)$, $1\leq g(D_1), g(D_2)$, and $D_3$ is a (maybe empty) set of parallel curves encircling $D_1 \cup D_2$. If $D_3$ is non empty there is an $h$ such that $1<h<g$ and $\varphi_h(D) = \varphi_h(D_1)+\varphi_h(D_2) +1$ and $\varphi_{h'}(D) = \varphi_{h'}(D_1)+\varphi_{h'}(D_2)$ for $1<h'<g$, $h'\neq h$. The number $h$ is the number of holes of the region bounded by the most internal component of $D_3$. If $D_3$ is emty $\varphi_h(D) = \varphi_h(D_1)+\varphi_h(D_2)$ for all $h$. We conclude using the inductive hypothesis on $D_1$ and $D_2$.

$(2)$ Clearly $g(D_s)\leq g(D) \leq g$, furthermore the Euler characteristic of the regions of the shadow $X_s$ is non positive. Therefore we have the inequalities on the left and right of the statement. For the the point $(1)$ and the third statement of Lemma~\ref{lem:psi} we get
$$
1-g(D_s) \leq \sum_{h=2}^{g(D_s)} (1-h) \varphi_h(D_s) = \sum_{h=1}^{g(D_s)} (1-h) \varphi_h(D_s) \leq  \sum_R \chi(R)\xi_0(R) = \psi(s)  .
$$
\end{proof}
\end{prop}

\begin{defn}
Enumerate the boundary components of $S_{(g)}$. Let $D\subset S_{(g)}$ be a link diagram and $s$ a state of $D$. We denote with $p_i(s)$ the number of components of the splitting of $D$ with $s$ that are parallel to the $i^{\rm th}$ boundary component of $S_{(g)}$ ($i=1,\ldots, g+1$), and with $p(s)$ the total number of homotopically non trivial components of the splitting of $D$ with $s$. 
\end{defn}
For $g\geq 3$ it is not true that $\sum_{i=1}^{g+1} p_i(s) =p(s)$, but it is true if $g=2$. The links represented by a diagram $D\subset S_{(2)}$ are $\Z_2$-homologically trivial if and only if $p_i(s)+p_j(s)\in 2\Z$ for any state $s$ and any $i,j$. If this holds, since $p(s)=p_1(s)+p_2(s)+p_3(s)$, for every $i$ we have that $p_i(s)\in 2\Z$ if and only if $p(s)\in 2\Z$. These simple properties furnish some more information about the Kauffman bracket of links in $\#_2(S^1\times S^2)= (S^1\times S^2)\# (S^1\times S^2)$ (we do not need it to prove the main theorem):
\begin{prop}\label{prop:state_sum2}
Fix a proper embedding of $S_{(2)}$ in a handlebody of the H-decomposition of $\#_2(S^1\times S^2)$. Let $D\subset S_{(2)}$ be a link diagram of a $\Z_2$-homologically trivial link in $\#_2(S^1\times S^2)$, and let $s$ be a state of $D$. Then
$$
\psi(s) = \begin{cases}
0 & \text{ if } p(s)\in 2\Z  \\
-1 & \text{ if } p(s) \in 2\Z+1
\end{cases}.
$$
\begin{proof}
If $p_i(s)=0$ for some $i=1,2,3$ ($p(s) \in 2\Z$) then $D_s$ is contained in the union of two disjoint annuli, $A_1$ and $A_2$, that encircle the boundary components different from the $i^{\rm th}$ one. Let $D_{s,j}$ be the diagram in the annulus $A_j$. Then by Remark~\ref{rem:tensor} $\langle D_s\rangle =\langle D_{s,1}\rangle \cdot \langle D_{s,2}\rangle $ where $D_{s,1}$ and $D_{s,2}$. Hence by \cite[Lemma 3.3]{Carrega_Tait1} $\langle D_s \rangle $ is a positive integer number and so $\psi(s) =0$.

Suppose $p_i(s) \neq 0$ for all $i=1,2,3$. The shadow $X_s$ has a region that is a 2-disk with two holes and the other ones are annuli. By Lemma~\ref{lem:psi} we can compute $\psi(s)$ using the admissible coloring $\xi_o$, $\psi(s) = \sum_R \chi(R) \cdot \xi_0(R)$. Hence $\psi(s)= -\xi_0(R')$, where $R'$ is the region that is a disk with two holes. For any admissible coloring $\xi$ of $X_s$ we have that the color $\xi(R')\in 2\Z+1$ if and only if $p(s)\in 2\Z+1$. We conclude since $\xi_0$ assigns just color $0$ and $1$.
\end{proof}
\end{prop}

Let $D\subset S_{(g)}$ be a $n$-crossing, connected diagram with $g(D)=g$ that represents a $\Z_2$-homologically trivial link in $\#_g(S^1\times S^2)$ by a proper embedding of $S_{(g)}$ in a handlebody of the H-decomposition of $\#_g(S^1\times S^2)$. Then for a state $s$ we get
\beq
\ord_\infty \langle D | s\rangle & = & \sum s(i) + 2( sD + \psi(s) ) \\  
\ord_0 \langle D | s\rangle & = & \sum s(i) - 2( sD + \psi(s) ) .
\eeq
Therefore
$$
\ord_\infty \langle D | s_+ \rangle - \ord_0 \langle D | s_- \rangle = 2(n + s_+D +s_-D +\psi(s_+) + \psi(s_-)) .
$$

The following is the analogue result of \cite[Lemma 4.1]{Carrega_Tait1} and is the part that needs most work to get the final result:
\begin{lem}\label{lem:ineq1_g}
Let $D\subset S_{(g)}$ be a $n$-crossing, connected diagram of a $\Z_2$-homologically trivial link in $\#_g(S^1\times S^2)$. Then
$$
B(\langle D \rangle) \leq \ord_\infty \langle D|s_+ \rangle - \ord_0 \langle D|s_- \rangle = 2(n + s_+D + s_-D + \psi(s_+) +\psi(s_-)) .
$$
Moreover if $D$ is adequate the equality holds:
$$
B(\langle D \rangle) = \ord_\infty \langle D | s_+ \rangle -\ord_0 \langle D|s_- \rangle .
$$
\begin{proof}
For every state $s$ we denote with $M(s)$ and $m(s)$ respectively the order in $\infty$ and in $0$ of $\langle D | s\rangle$. Hence
$$
B(\langle D \rangle) \leq \max_s M(s) - \min_s m(s) ,
$$
and the equality holds if there is a unique maximal $s$ for $M$ and a unique minimal $s$ for $m$. We have:
\beq
M(s) & = & \sum_i s(i) +2sD +2\psi(s) \\
m(s) & = & \sum_i s(i) -2sD -2\psi(s) .
\eeq
Let $s$ and $s'$ be two states differing only in a crossing $j$ where $s(j)=+1$ and $s'(j)=-1$. We have
\beq
M(s) - M(s') & = & 2(1+sD-sD' +\psi(s) - \psi(s')) \\
m(s) - m(s') & = & 2(1-sD+sD' -\psi(s) + \psi(s')) .
\eeq
After the splitting of every crossing different from $j$ we have some cases up to mirror image. Four of them are divided in three types and are shown in Fig.~\ref{figure:cases_lem} (we pictured just an homotopically non trivial annulus where the diagrams hold): the crossing $j$ lies on a component contained in an annulus, $s'D=sD \pm 1$. In the remaining cases $j$ lies on a component that is not contained in an annulus, $s'D = sD$, they are shown in Fig.~\ref{figure:cases_lem2} (we pictured just a disk with two holes that is not homotopically equivalent to an annulus where the diagrams hold). Now we better examine the four types of case in order to get information about $\psi(s)-\psi(s')$, and get $M(s)\geq M(s')$ and $m(s)\geq m(s')$.

In the first two types of case (Fig.~\ref{figure:cases_lem}-(left) and Fig.~\ref{figure:cases_lem}-(center)) $D_{s'}=D_s$, hence $\psi(s')=\psi(s)$. Therefore $M(s)-M(s')= 2\pm 2 \geq 0$, and $m(s)-m(s') = 2\pm 2\geq 0$. 

In the third type of case (Fig.~\ref{figure:cases_lem}-(right)) $D_{s'}$ and $D_s$ differ by the removal or the addition of an annulus region that produces the fusion or the division of previous regions. Suppose we are in the third case and the splitting with $s'$ is obtained fusing two components of the splitting with $s$ (the other case is analogous). Then $D_{s'}$ is obtained from $D_s$ removing two parallel components. The regions of $D_s$ that do not touch the selected parallel components remain unchanged in $D_{s'}$. The selected components of $D_s$ bound an annulus region $A$. The annulus $A$ touches two different regions, $R_1$ and $R_2$, that are disks respectively with $h_1>0$ and $h_2>0$ holes. In $D_{s'}$ the regions $A$, $R_1$ and $R_2$ are replaced with a region $R$ that is a disk with $h_1+h_2-1$ holes. Let $\xi_0$ and $\xi'_0$ be the admissible colorings of $X_s$ and $X_{s'}$ that extend the boundary coloring and assign only color $0$ and $1$ (see Lemma~\ref{lem:psi}). We have two cases:
\begin{enumerate}
\item{$\xi_0(R_1)=\xi_0 (R_2) =\xi'_0(R) =0$ and $\xi(A)=1$;}
\item{$\xi_0(R_1)=\xi_0 (R_2) =\xi'_0(R) =1$ and $\xi(A)=0$.}
\end{enumerate}
Since $\chi(A)=0$, $\chi(R) = \chi(R_1)+\chi(R_2)$, $\psi(s)= \sum_{\bar R} \chi(\bar R) \xi_0(\bar R)$ and $\psi(s')= \sum_{\bar R} \chi(\bar R) \xi'_0(\bar R)$, we get that $\psi(s)=\psi(s')$. Therefore $M(s)-M(s')= 2\pm 2\geq 0$ and $m(s)- m(s') =2\pm 2 \geq 0$.

Suppose we are in the fourth type of case (Fig.~\ref{figure:cases_lem2}) and the splitting with $s'$ is obtained dividing in two a component of the splitting with $s$ (the other case is analogous). In the same way we get $D_{s'}$ from $D_s$. We have that the selected component of $D_s$ bounds on a side a disk with $h_1+h_2$ holes, $R_1$, and on the other one a disk with $k\geq 2$ holes, $R_2$, while those two components of $D_{s'}$ together bound an internal region with $k+1$ holes, $R'_2$, and two different punctured disks with $h_1>0$ and $h_2$ holes, respectively $R'_{1,1}$ and $R'_{1,2}$. Again, let $\xi_0$ and $\xi'_0$ be the admissible colorings of $X_s$ and $X_{s'}$ that extend the boundary coloring and assign only the colors $0$ and $1$ (see Lemma~\ref{lem:psi}). We have two cases:
\begin{enumerate}
\item{$\xi_0(R_1)=\xi'_0(R'_{1,1})= \xi'_0(R'_{1,2}) =0$ and $\xi_0(R_2)=\xi'_0(R'_2)=1$;}
\item{$\xi_0(R_1)=\xi'_0(R'_{1,1})= \xi'_0(R'_{1,2}) =1$ and $\xi_0(R_2)=\xi'_0(R'_2)=0$.}
\end{enumerate}
Since $\psi(s)= \sum_{\bar R} \chi(\bar R)\xi_0(\bar R)$ and $\psi(s') = \sum_{\bar R} \chi(\bar R) \xi'_0(\bar R)$ we get $\psi(s')=\psi(s) -1$ in the first case, and $\psi(s')=\psi(s)$ in the second one. Therefore in each case $M(s)-M(s')\geq 0$ and $m(s)- m(s') \geq 0$.

Given a state $s$ we can connect it to $s_+$ finding a sequence of states $s_+=s_0 , s_1 \ldots , s_k =s$ such that $s_r$ differs from $s_{r+1}$ only in a crossing and $\sum_i s_r(i) = \sum_i s_{r+1}(i) +2$. Analogously for $s_-$. Hence $M(s) \leq M(s_+)$ and $m(s) \geq m(s_-)$. Thus we have the first statement. 

If the diagram is plus-adequate (resp. minus-adequate) the fourth type of case (Fig.~\ref{figure:cases_lem2}) is not allowed and it holds $M(s_+) - M(s)=4$ ($m(s) - m(s_-) =4$) for each $s$ such that $\sum_s s(i)= n-2$ ($= 2-n$). Namely we have a strict inequality already from the beginning of the sequence $s_0,\ldots, s_k$. This proves the second statement.
\end{proof}
\end{lem}

\begin{figure}[htbp]
\begin{center}
\parbox[c]{3cm}{ 
\includegraphics[scale=0.4]{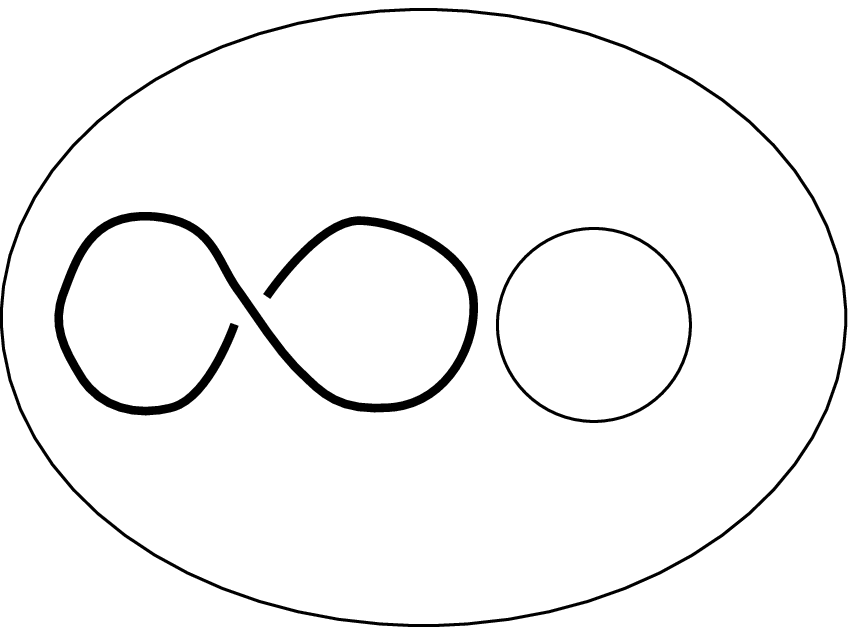} \\
\includegraphics[scale=0.4]{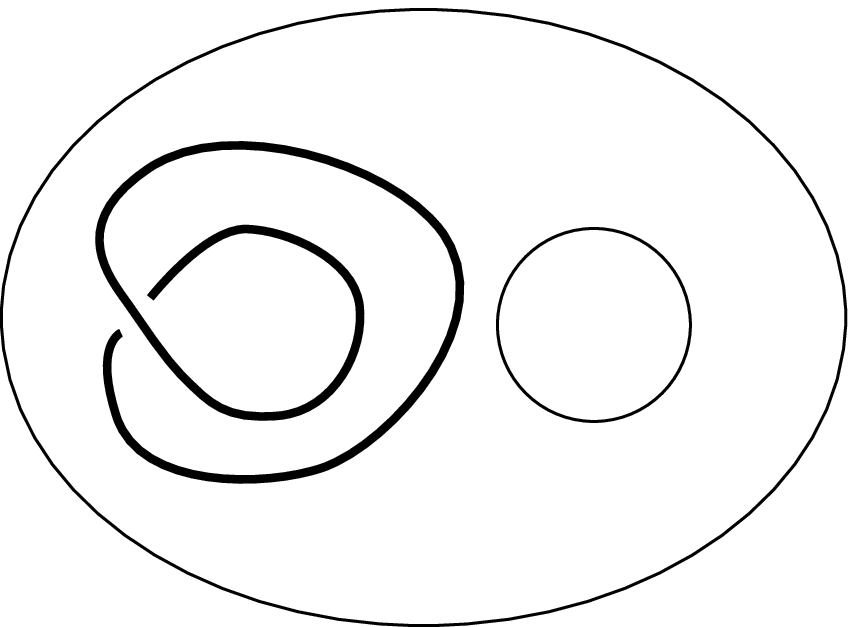} }
\hspace{0.35cm}
\parbox[c]{3cm}{  \includegraphics[scale=0.4]{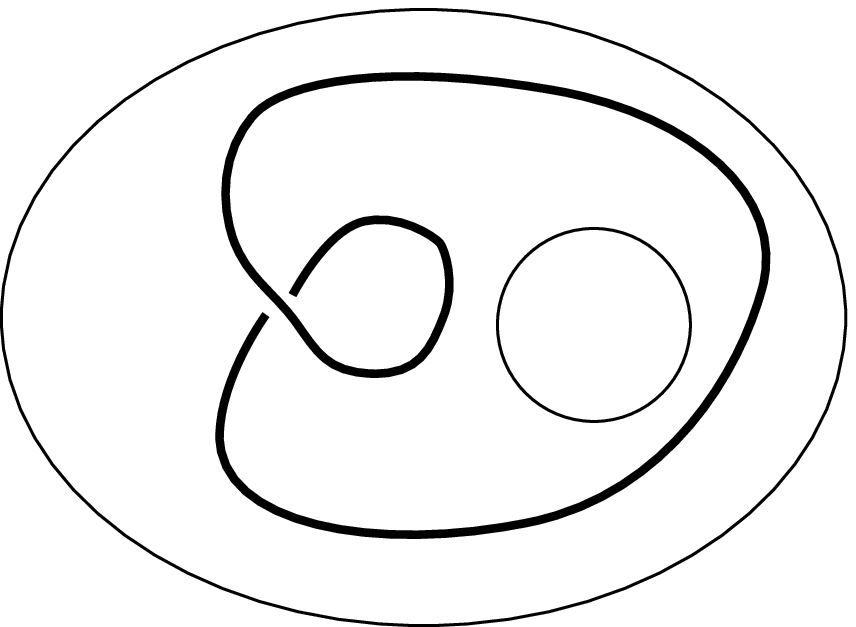} }
\hspace{0.35cm}
\parbox[c]{3cm}{ \includegraphics[scale=0.4]{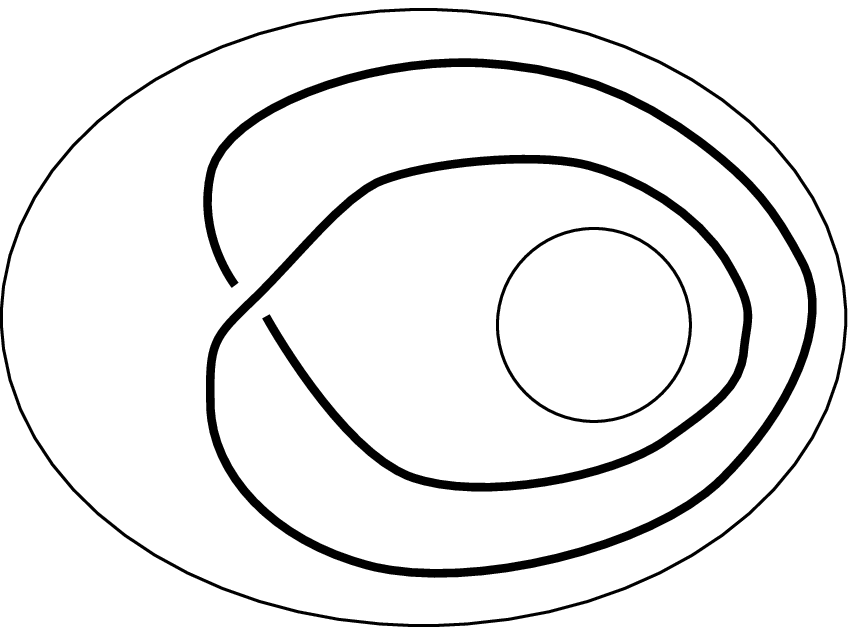} }
\end{center}
\caption{The first three types of cases of the proof of Lemma~\ref{lem:ineq1_g}.}
\label{figure:cases_lem}
\end{figure}

\begin{figure}[htbp]
\begin{center}
\includegraphics[scale=0.4]{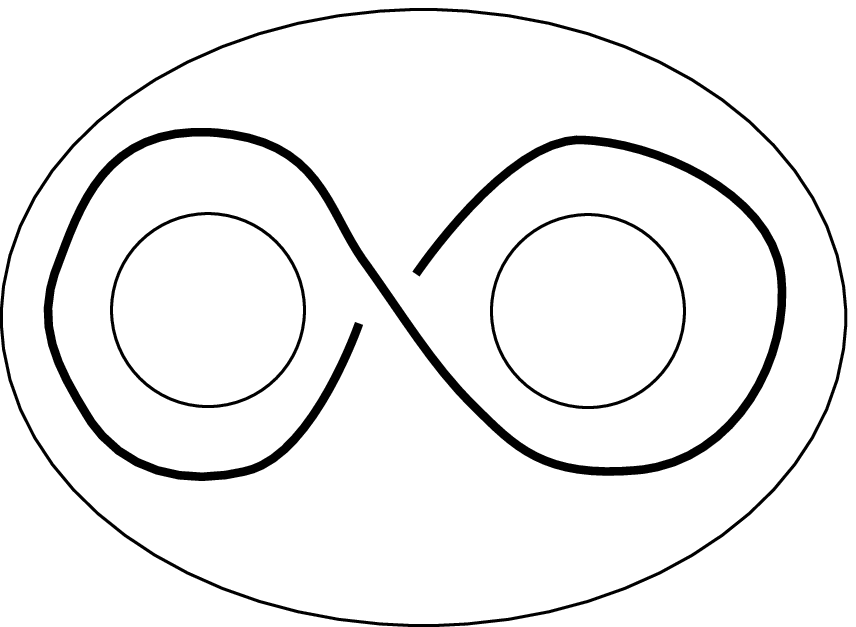}
\hspace{0.5cm}
\includegraphics[scale=0.4]{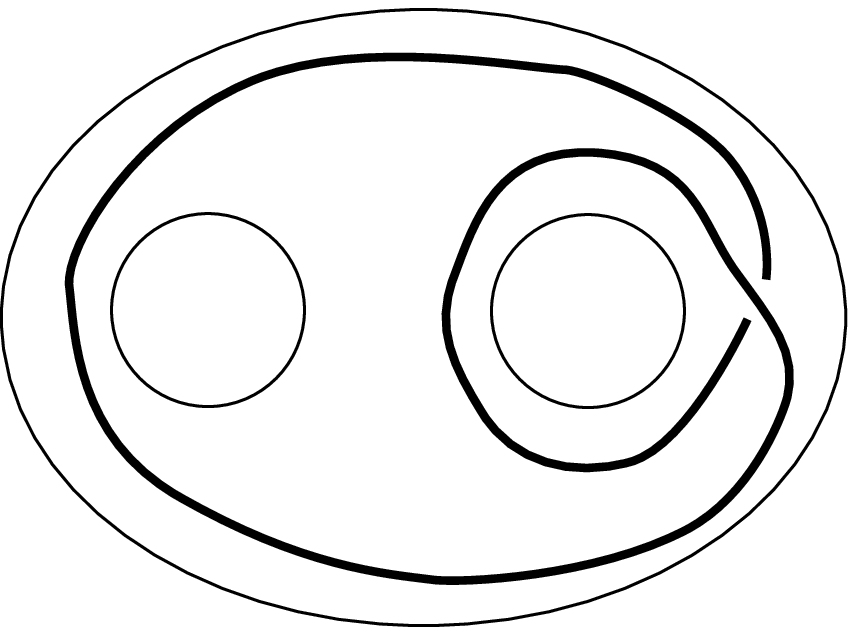}
\end{center}
\caption{The fourth type of case of the proof of Lemma~\ref{lem:ineq1_g}.}
\label{figure:cases_lem2}
\end{figure}

The following comprehends the analogue result of \cite[Lemma 4.2]{Carrega_Tait1}:
\begin{lem}\label{lem:ineq2_g}
Let $D\subset S_{(g)}$ be a $n$-crossing, connected diagram with $g(D)=g$ of a $\Z_2$-homologically trivial link in $\#_g(S^1\times S^2)$. Then
$$
s_+ D + s_- D \leq n + 1-g  .
$$
\begin{proof}
We proceed by induction on $n$. We have $g \leq n $, and, up to homeomorphism of $S_{(g)}$, there are exactly $2^g$ diagrams with these characteristics (connected, $g(D)=g$, and representing a $\Z_2$-homologically trivial link) and with $g$ crossings: the ones described in Fig.~\ref{figure:n=g}. Now we prove the statement for the diagrams with $g$ crossings (the base for the induction on $n$) by induction on $g$. 

For $g=0$ we have just a homotopically trivial circle and its Kauffman bracket is $-A^2-A^{-2}$, hence the statement holds. Suppose that the statement is true for such diagrams in $S_{(g-1)}$ with $g-1$ crossings. We split the crossing on the right of Fig.~\ref{figure:n=g}. The result is a diagram obtained taking a one of those diagrams $D'$ in $S_{(g-1)}$ and either adding an hole in the external region surrounding the hole on the right (according to the figure) of $S_{(g-1)}$, or adding a hole in the internal region and surrounding it with a circle. We did not add any homotopically trivial components, hence
$$
s_+D +s_-D =  s_+D' +s_-D' \leq 1 .
$$

Now suppose that the statement is true for all diagrams in $S_{(g)}$ with less than $n>g$ crossings. We claim, and prove later, that there is a crossing $j$ of $D$ and a splitting $D'$ of $j$ with the same characteristics of $D$ (connected and $g(D')=g$). We can suppose that the splitting of $j$ is done in the positive way (the negative way is analogous), hence $s_+D' = s_+D$ and $s_-D -1 \leq s_-D' \leq s_-D + 1$. Therefore
\beq
s_+D + s_-D & \leq  & s_+D' + s_- D' +1 \\
& \leq & (n-1)+1-g +1 \\
& = & n+1-g .
\eeq

Finally we prove the claim. Since $D$ is connected every crossing has a splitting that is still connected. If there is a crossing not adjacent to two external regions, every splitting $D'$ of the crossing satisfies $g(D')=g$. Since $n>g$, we get the previous case if $D$ has no more than $g$ crossings adjacent to two external regions. If there are at least $g+1$ crossings adjacent to two external regions, we take as $j$ one of them, its connected splitting $D'$  satisfies $g(D')=g$. 
\end{proof}
\end{lem}

\begin{figure}[htbp]
\begin{center}
\includegraphics[width=7cm]{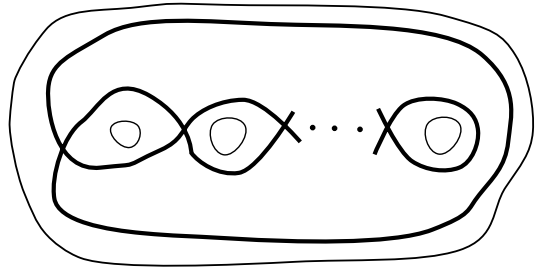}
\end{center}
\caption{Up to homeomorphism of $S_{(g)}$ all the link diagrams $D$ of $\Z_2$-homologically trivial links in $\#_g(S^1\times S^2)$ with $g$ crossings and $g(D)=g$ are obtained choosing the over/underpass for every crossing of the graph in figure.}
\label{figure:n=g}
\end{figure}

The following comprehends the analogue result of \cite[Lemma 4.3]{Carrega_Tait1}:
\begin{lem}\label{lem:alter_eq_g}
Let $D\subset S_{(g)}$ be a $n$-crossing alternating connected link diagram with $g(D)= g$ that represents a $\Z_2$-homologically trivial link in $\#_g(S^1\times S^2)$. Then 
$$
\ord_\infty \langle D|s_+ \rangle - \ord_0 \langle D|s_- \rangle = 4n +4 -4g .
$$
\begin{proof}
The number of edges of $D$ as a planar graph is $2n$, hence with a computation of Euler characteristic of $S_{(g)}$ we get that the sum of the Euler characteristics of the regions is $n+1-g$. All the internal regions are disks and there are $g+1$ external regions that are annuli. Thus there are $n+1-g$ internal regions. The diagram $D_{s_\pm}$ has a region that is a disk with $g$ holes, and the others are annuli. Thus
$$
\langle D|s_\pm \rangle = A^{\pm n} (-A^2 -A^{-2})^{s_\pm D + 1 - g} .
$$
The regions of $D_{s_\mp}$ are all disks. Thus
$$
\langle D|s_\mp \rangle = A^{\mp n} (-A^2 -A^{-2})^{s_\mp D}  .
$$
Since the link is alternating, the number of internal regions is equal to $s_+D+s_-D$. Therefore
\beq
\ord_\infty \langle D|s_+ \rangle - \ord_0 \langle D|s_- \rangle & = &  2(n + s_+ D + s_- D + 1 - g) \\
& = & 2(n+n+2-p(s_+ )- p(s_-) +1 -g) \\
& = & 4n + 4 -4g .
\eeq
\end{proof}
\end{lem}

\begin{proof}[Proof of Theorem~\ref{theorem:Tait_conj_g}]
Every diagram $\bar D\subset S_{(g)}$ of $L$ for any embedding of $S_{(g)}$ must be connected and with $g(\bar D)=g$. By Proposition~\ref{prop:reduced_D_g}, Lemma~\ref{lem:ineq1_g}, Lemma~\ref{lem:ineq2_g} and Lemma~\ref{lem:alter_eq_g} we get
\beq
4n(D)+4 -4g & = & B(\langle D \rangle) \\
& = & B(\langle D' \rangle) \\
& \leq & 4n(D') +2 -2g +2(\psi(s_+,D') +\psi(s_-,D')) \\
& \leq & 4n(D') +2 -2g ,
\eeq
where $\psi(s_\pm,D')$ is the quantity $\psi(s_\pm)$ related to the diagram $D'$. Therefore $n(D)\leq n(D') + (g-1)/2$. If $g\leq 2$, $n(D)\leq n(D') +\frac 1 2$, but $n(D)$ and $n(D')$ are integers, hence $n(D)\leq n(D')$. 
\end{proof}

\begin{proof}[Proof of Theorem~\ref{theorem:Tait_conj_Jones_g}]
We have already proved the case $k=0$ in Lemma~\ref{lem:alter_eq_g}. Suppose $k>0$. We proceed by induction on $g$. We apply the second identity of Fig.~\ref{figure:sphere} (with $a=b=1$) near one such crossing. We get $\langle D \rangle= \langle D' \rangle /(-A^2-A^{-2})$, where $D'\subset S_{(g)}$ is a $n$-crossing, connected, alternating diagram with $g(D')=g-1$ and $0<h\leq k$ crossings adjacent twice to a region. We have that $\langle D' \rangle = \langle D''\rangle$, where $D''$ is a link diiaigram in $S_{(g-1)}$ with the same characteristics of $D'$. Moreover we can easily get a diagram $D'''\subset S_{(g-1)}$ representing the same link of $D''$ and is connected, alternating, with $g(D''')=g-1$ and has $n-h$ crossings. Therefore by the inductive hypothesis 
\beq
B(\langle D \rangle) & = & B(\langle D' \rangle) -4 \\
& = & B(\langle D''' \rangle) -4 \\
& = & 4(n-h) +4 - 4(g-1)-4 (k-h)-4 \\
& = & 4n +4 -4k .
\eeq
\end{proof}

\begin{cor}\label{cor:conj_Tait_Jones_g}
Let $L\subset \#_g(S^1\times S^2)$ be a non H-split $\Z_2$-homologically trivial link. 
\begin{enumerate}
\item{If $B(\langle L \rangle)$ is not a multiple of $4$, then $L$ is not alternating.}
\item{If $B(\langle L \rangle)$ is not a positive multiple of $4$, then $L$ is not alternating with a simple diagram.}
\item{If $L$ has homotopic genus $g$, $L$ does not intersect any non separating 2-sphere and $B(\langle L \rangle)< 4n+4-4g$, then either $L$ is not alternating, or $L$ has crossing number lower than $n$.}
\end{enumerate}
\begin{proof}
It follows from Theorem~\ref{theorem:Tait_conj_Jones_g}. To get the second statement we must note that every connected diagram in $S_{(g)}$ that is not contained in a disk with $g'<g$ holes has at least $g$ crossings.
\end{proof}
\end{cor}

\begin{ex}\label{ex:no_alt_g}
The knots represented by the diagrams in Fig.~\ref{figure:ex_no_alt} ($g=1$) are $\Z_2$-homologically trivial and have Kauffman bracket equal to $0$ (left) and $A-A^{-3}-A^{-5}$ (right). Therefore by Corollary~\ref{cor:conj_Tait_Jones_g}-(2) they are not alternating. 

By Corollary~\ref{cor:conj_Tait_Jones_g}-(1) the link in Example~\ref{ex:Kauff}-(7) ($g=2$) is non alternating. By Corollary~\ref{cor:conj_Tait_Jones_g}-(3) the knot in Example~\ref{ex:Kauff}-(5) ($g=2$) either is non alternating or has crossing number lower than $4$. For all the links in $\#_2(S^1\times S^2)$ with crossing number $3$ there is a non separating sphere that intersects the link at most twice. The knot in Example~\ref{ex:Kauff}-(5) does not have one such sphere, hence it is not alternating. 

Unfortunately we are not able to say if the links in Example~\ref{ex:Kauff}-(6) and Example~\ref{ex:Kauff}-(11) ($g=2$) are not alternating. Using Theorem~\ref{theorem:Tait_conj_Jones_g} we can only say that if the knot in Example~\ref{ex:Kauff}-(6) was alternating it should have an alternating diagram whose crossings are all adjacent to two regions ($n=k$).
\end{ex}

\begin{figure}
\begin{center}
\includegraphics[scale=0.55]{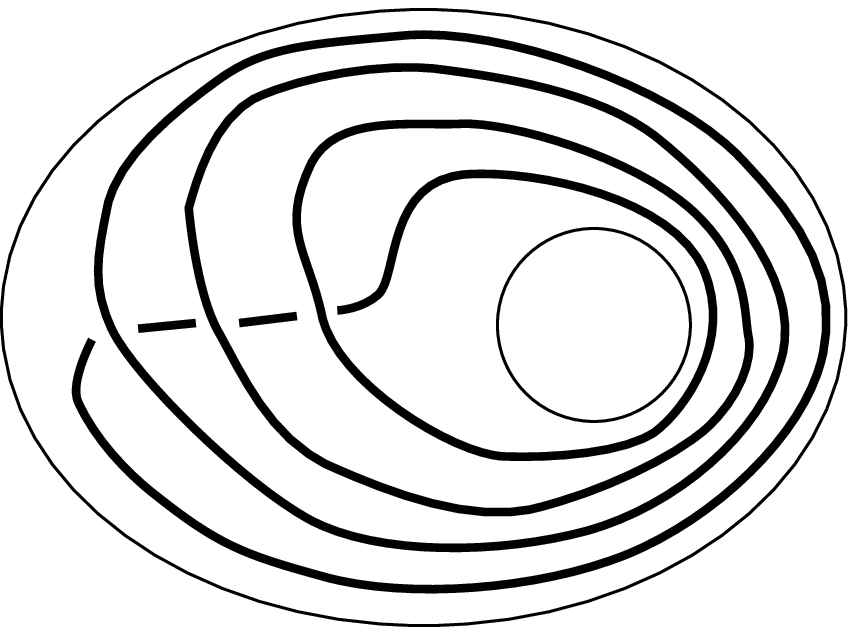}
\hspace{0.5cm}
\includegraphics[scale=0.55]{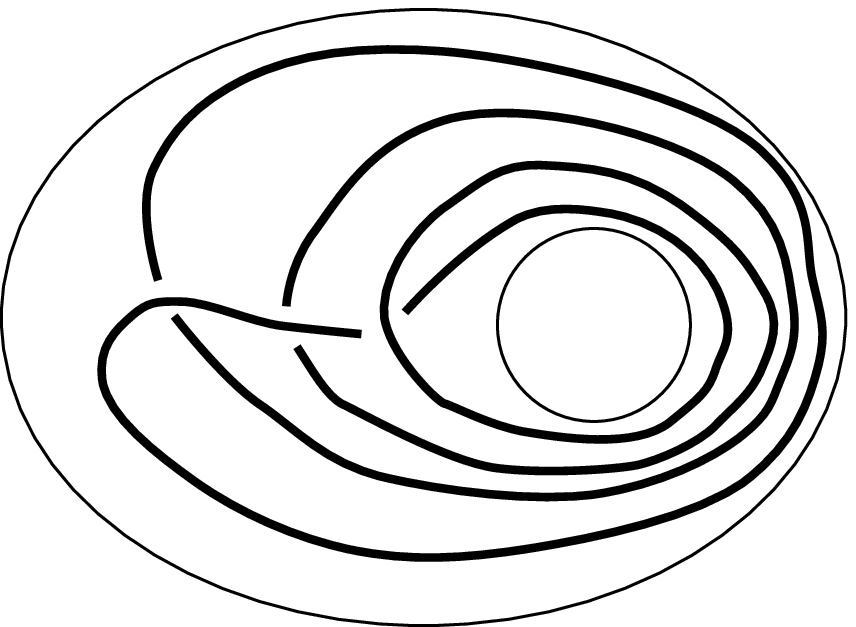}
\end{center}
\caption{Two $\Z_2$-homologically trivial knots in $S^1\times S^2$ whose Kauffman bracket are $0$ (left) and $A-A^{-3}-A^{-5}$ (right).}
\label{figure:ex_no_alt}
\end{figure}

\section{Open questions}
In this final section we ask some open questions whose solution in the positive would produce a sharper, or more complete, result than Theorem~\ref{theorem:Tait_conj_g}, and and we provide some information that
suggests that the natural extension of the Tait conjecture in $\#_g(S^1\times S^2)$ could be false for $g\geq 3$. 

The first obvious question is the natural extension of the Tait conjecture:
\begin{quest}\label{quest:Tait_conj_g}
Fix a proper embedding of $S_{(g)}$ in a handlebody of the H-decomposition of $\#_g(S^1\times S^2)$. Let $D\subset S_{(g)}$ be a connected, alternating, simple diagram of a $\Z_2$-homologically trivial link in $\#_g(S^1\times S^2)$ that is non H-split and with homotopic genus $g$. Is the number of crossings of $D$ equal to the crossing number of $L$?
\end{quest}

The following two questions aim to remove from Theorem~\ref{theorem:Tait_conj_g} the hypothesis of ``homotopic genus $g$'' and ``non H-split''.
\begin{quest}
Let $D\subset S_{(g)}$ be an alternating connected link diagram. Once a proper embedding of $S_{(g)}$ in a handlebody of the H-decomposition of $\#_g(S^1\times S^2)$ is fixed, is the link represented by $D$ non H-split?
\end{quest}

\begin{quest}
Let $D\subset S_{(g)}$ be a connected, alternating link diagram. Is the homotopic genus of the links represented by $D$ equal to $g(D)$?
\end{quest}

\begin{rem}\label{rem:ineq_psi}
If for every connected diagram $D\subset S_{(g)}$ with $g(D)=g$ and representing a $\Z_2$-homologically trivial link in $\#_g(S^1\times S^2)$, we had
$$
\psi(s_+) + \psi(s_-) \leq  2-g ,
$$
we could make Theorem~\ref{theorem:Tait_conj_g} sharper and answer positively to Question~\ref{quest:Tait_conj_g}. Unfortunately it is not true:  a diagram $D\subset S_{(3)}$ is shown in Example~\ref{ex:Kauff}-$(7)$, $D$ is simple, connected, with $g(D)=3$, represents a $\Z_2$-homologically trivial link, and it satisfies $\langle D | s_+ \rangle = A^7(-A^2-A^{-2})$, $\langle D|s_-\rangle = A^{-7}$, hence $\psi(s_+)=\psi(s_-)=0 >2-g$.
\end{rem}

In Remark~\ref{rem:ineq_psi} we did not require that the diagrams represent alternating links, and we know that the shown example represents a non alternating link (see Example~\ref{ex:no_alt_g}). So, it is natural to ask the following: 

\begin{quest}
Let $D\subset S_{(g)}$ be a connected (maybe not alternating) diagram with $g(D)=g$ that represents an alternating $\Z_2$-homologically trivial link in $\#_g(S^1\times S^2)$. Is it true that 
$$
\psi(s_+) + \psi(s_-) \leq 2-g \ \ ?
$$
\end{quest}


\begin{thebibliography}{5}


\bibitem[BFK]{BFK} D. Bullock, C. Frohman, J. Kania-Bartoszynska, \emph{The Yang-Mills measure in the Kauffman bracket skein module},  Comment. Math. Helv. 78 (2003), 1--17.

\bibitem[Ca]{Carrega_Tait1} A. Carrega, \emph{The Tait conjecture in $S^1\times S^2$}, {\tt  	arXiv:1510.01622} (2015).

\bibitem[CaMa]{Carrega-Martelli} A. Carrega and B. Martelli, \emph{Shadows, ribbon surfaces, and quantum invariants}, {\tt  arXiv:1404.5983} (2014).

\bibitem[Carv]{Carvalho} L.N. Carvalho, \emph{A classification of automorphisms of compact 3-manifolds}, {\tt arXiv:math/0510610} (205).

\bibitem[Co0]{Costantino0} F. Costantino, \emph{A short introduction to shadows of 4-manifolds}, {\tt arXiv:math/0405582} (2004).

\bibitem[Co1]{Costantino2} F. Costantino, \emph{Colored Jones invariants of links in $\#_k S^2\times S^1$ and the Volume Conjecture}, J. Lond. Math. Soc. 76 (2007) 1--15.

\bibitem[H]{Hatcher} A. Hatcher, \emph{Notes on basic 3-manifold topology}, {\tt https://www.math.cornell.edu/~hatcher/3M/3Mfds.pdf} (2007).

\bibitem[HP]{HP2} J. Hoste and J. Przytycki, \emph{The Kauffman bracket skein module of $S^1\times S^2$}, Mathematische Zeitschrif, 220 (1995), 65--73.

\bibitem[Ka]{Kauffman_Tait} L.H. Kauffman, \emph{State models and the Jones polynomial}, Topology 26 (1987), 395--407. 

\bibitem[Li]{Lickorish} W.B.R. Lickorish, ``An introduction to knot theory'', Graduate Texts in Mathematics 175, Springer-Verlag, New York (1997).

\bibitem[Mu1]{Murasugi1} K. Murasugi, \emph{The Jones polynomial and classical conjectures in knot theory}, Topology 26 (1987), 187--194.

\bibitem[Mu2]{Murasugi2} K. Murasugi, \emph{The Jones polynomial and classical conjectures in knot theory II}, Math. Proc. Cambridge Philos. Soc. 102 (1987), 317--318.

\bibitem[Pr1]{Pr1} J.H. Przytycki, \emph{Fundamentals of Kauffman Bracket Skein Modules},
{\tt arXiv:math.GT/9809113} (1998).

\bibitem[Pr2]{Pr2} J.H. Przytycki, \emph{Kauffman bracket skein module of a connected sum of 3-manifolds},  Manuscripta Math. 101 (2000), 199--207.

\bibitem[S]{Schultens} J. Shultens, \emph{Heegaard -zerlegungen der 3-sph\"are}, {\tt https://www.math.ucdavis.edu/~jcs/pubs/Waldd1.pdf} (2008).

\bibitem[Thi]{Thistlethwaite} M.B. Thistlethwaite, \emph{Kauffman's polynomial and alternating links},  Topology 27 (1988), 311--318.

\bibitem[Tait]{Tait} P..G. Tait, ``On knots I, II, III'', Sientific Papers 1, Cambridge University Press, London (1898), 273--347.

\bibitem[Tu]{Turaev} V.G. Turaev, ``Quantum invariants of knots and 3-manifolds'', Walter de Gruyter (1994).

\bibitem[W]{Waldhausen} F. Waldhausen, \emph{Heegaard -zerlegungen der 3-sph\"are}, Topology 7 (1968), 195--203.

\end{thebibliography}
\end{document}